\documentclass[10.9pt]{amsart}
\usepackage{amsmath}
\usepackage{amsmath,accents}
 \usepackage{amscd}
 \usepackage{amssymb}
\usepackage[applemac]{inputenc}
\usepackage{tikz-cd}
\usepackage{ulem}
\usepackage[applemac]{inputenc}
\usepackage[all]{xy}
\usepackage{color}

\usepackage{xparse}

\NewDocumentCommand{\dslash}{s}{%
  \IfBooleanTF{#1}
    {\big/\mkern-7mu\big/}
    {/\mkern-6mu/}%
}

\newtheorem{definition}{Definition}[section]

\newtheorem{proposition}{Proposition}[section]

\newtheorem{lemma}[proposition]{Lemma}
\newtheorem{theorem}[proposition]{Theorem}

\newtheorem{remark}{Remark}[section]

\DeclareMathOperator{\Ker}{Ker}
\DeclareMathOperator{\Supp}{Supp}
\DeclareMathOperator{\Sing}{Sing}
\DeclareMathOperator{\Aut}{Aut}

\DeclareMathOperator{\Pic}{Pic}

\DeclareMathOperator{\Spec}{Spec}
\DeclareMathOperator{\rank}{rank}

\title{ \small \text{Views on level {$\mathit \ell$} curves, K3 surfaces and Fano threefolds} }
\author[ A. Garbagnati]{Alice Garbagnati}
\address{Dipartimento di Matematica 
Universit\'a degli Studi di Milano 
 }
  \email{alice.garbagnati@unimi.it }

\author[   A. Verra]{Alessandro Verra}
  \address{Dipartimento di Matematica \\ 
Universit\'a degli Studi di Roma Tre \\ 
 }
 \email{sandro.verra@gmail.com}
 \thanks{This work was partially supported by INdAM-GNSAGA and by the project  PRIN-2017 'Moduli Theory and Birational Classification'}
\begin{document}
 \maketitle 
 \begin{abstract}  An analogue  of the Mukai map $m_g: \mathcal P_g \to \mathcal M_g$ is studied for the moduli $\mathcal R_{g, \ell}$ of genus $g$ curves $C$ with a level $\ell$ structure. Let $\mathcal P^{\perp}_{g, \ell}$ be the moduli space of $4$-tuples $(S, \mathcal L, \mathcal E, C)$ so that $(S, \mathcal L)$ is a polarized K3 surface of genus $g$, $\mathcal E$ is orthogonal to $\mathcal L$ in $\Pic S$ and defines a standard degree $\ell$ K3 cyclic cover of $S$, $C \in \vert \mathcal L \vert$. We say that $(S, \mathcal L, \mathcal E)$ is a level $\ell$ K3 surface. These exist for $\ell \leq 8$ and their families are known. We define a level $\ell$ Mukai map $r_{g, \ell}: \mathcal P^{\perp}_{g, \ell} \to \mathcal R_{g, \ell}$,  induced by the assignment of $(S, \mathcal L, \mathcal E, C)$ to $ (C, \mathcal E \otimes \mathcal O_C)$.  We investigate a curious possible analogy between $m_g$ and $r_{g, \ell}$, that is, the failure of the maximal rank of $r_{g, \ell}$ for $g = g_{\ell} \pm 1$, where $g_{\ell}$ is
 the value of $g$ such that $\dim \mathcal P^{\perp}_{g, \ell} = \dim \mathcal R_{g,\ell}$. This is proven here for $\ell = 3$. As a related open problem we discuss Fano threefolds whose hyperplane sections are level $\ell$ K3 surfaces and their classification.

 \end{abstract}
 \section{Introduction}
Our aim is to convince the reader, showing a program and new results,  of the interest represented by some complex projective varieties whose curvilinear sections are canonical curves $C$ of genus $g$, endowed with a distinguished  nonzero  $\ell$-torsion element $\eta \in \Pic C$. 
 Often one says that $(C, \eta)$ is a \it level $\ell$ curve of genus $g$\rm, cfr. \cite{CEFS}. Fixing $(g, \ell)$ the moduli space of these pairs is integral, quasi projective and  denoted by  $\mathcal R_{g, \ell}$. 
 \par To enter further in the matter let us mention two other names from the title: \it $K3$ surface \rm and \it Fano threefold. \rm The $K3$ surfaces $S$ we consider are very special: they admit a non split cyclic cover of degree $\ell$, still birational to a K3 surface. This is defined by a line bundle $\mathcal O_S(E):= \mathcal E$ such that $h^0(\mathcal O_S(\ell E)) = 1$ and $h^0(\mathcal O_S(mE)) = 0$ for $m < \ell$.  The study of these surfaces stems from Nikulin's classification of K3 surfaces with an order $\ell$ symplectic automorphism and the classification implies $\ell \leq 8$, \cite{N1}. Since then several foundational results, in use here, did follow, cfr. \cite{ GP, GS1, G, GS2,  vGS}.  \par Now let $\mathcal L \in \Pic S$ be a genus $g$ polarization orthogonal to $\mathcal E$. Let $\eta := \mathcal O_C(E)$, where $C \in \vert \mathcal L \vert$ is smooth, then it turns out that $(C, \eta)$ is a level $\ell$ curve. We say that the triple $(S, \mathcal L, \mathcal E)$ is a  \it level $\ell$ K3 surface of genus $g$, \rm see definition (\ref{defnat}) for some precision.  Fixing $\ell$ the moduli of these triples are reducible for  infinitely many values of $g$. However a distinguished irreducible component exists for every $g$,  namely  the moduli space of triples $(S, \mathcal L, \mathcal E)$ such that $\Pic S$ is the sum of $\mathbb Z\mathcal L$ and its orthogonal lattice.  We  denote it by 
 \begin{equation} \mathcal F^{\perp}_{g, \ell}. \end{equation}
 Finally we come to the moduli space $\mathcal P^{\perp}_{g, \ell}$ of $4$-tuples $(S, \mathcal L, \mathcal E, C)$ such that $C \in \vert \mathcal L \vert$ and $(S, \mathcal L, \mathcal E)$ defines a point in ${\mathcal F}^{\perp}_{g, \ell}$.  Such a space is strictly related with the first topic considered in our paper. To introduce it let us define the 
\it level $\ell$ Mukai map. \rm  This is  the rational map 
\begin{equation} \label{lmukai}
r_{g, \ell}: \mathcal P^{\perp}_{g, \ell} \to \mathcal R_{g, \ell},
\end{equation}
 assigning the moduli point of the $4$-tuple $(S, \mathcal L, \mathcal E, C)$ to the moduli point of the pair $(C, \eta)$,  where $\eta$ is $\mathcal O_C(E)$.  Let $\mathcal P_g$ be the moduli space of triples $(S, \mathcal L, C)$, where $(S, \mathcal L)$ is a polarized K3 surface of genus $g$ and $C \in \vert \mathcal L \vert$,  then the previous name is motivated by the well known  Mukai map
\begin{equation} \label{mukai}
m_g: \mathcal P_g \to \mathcal M_g,
\end{equation}
 assigning the moduli point of the triple $(S, \mathcal L, C)$ to the moduli point of the curve $C$.  Some famous connections between canonical curves of genus $g$, K3 surfaces and Fano threefolds are well represented by $m_g$ and, in particular, by a curious variation of its rank. We recall that a rational map $f: X \to Y$ of integral varieties has \it maximal rank \rm if $\dim f(X) = \min \lbrace \dim X, \dim Y \rbrace$. \par Considering $m_g$ we recall that $\dim \mathcal P_g = 19 + g$ and $\dim \mathcal M_g = 3g-3$, therefore $\dim \mathcal P_g = \dim \mathcal M_g$ iff $g = 11$. Now $m_{11}$ is birational but, curiously, $m_g$ fails to be of maximal rank precisely before and after this transition value, that is, for $g = 11 \pm 1$. For the rest
$m_g$ is dominant for $g \leq 9$ and generically injective for $g \geq 13$. As is well known this anomaly is due to the presence behind the scene of some Fano varieties, whose curvilinear sections are general canonical curves of genus $11 \pm 1$, cfr.  \cite{CLM, Mu, Mu1, S}.\par A  task of this paper is to point out the same possible anomalies for the level $\ell$ Mukai maps $r_{g, \ell}$. The case $\ell = 2$ has already been done and it is an experimental origin to this work. If $\ell = 2$  we have $\dim \mathcal P^{\perp}_{g,2} =  \dim \mathcal R_{g,2}$ for $g = 7$. Then $r_{g,2}$ fails to be of maximal rank for $g = 7\pm 1$ and is birational for $g = 7$, \cite{FV, KLV1, Ve}. The 'Fano varieties behind the scene' for $g = 8$ and $g = 6$ are addressed or revisited in section 7. \par
In section 5 we summarize the question for each $\ell$. Let $g_{\ell}$ be the unique value of $g$ such that $\dim \mathcal P^{\perp}_{g, \ell} = \dim \mathcal R_{g, \ell}$, for $ l  = 2, \ 3, \ 4, \ 5, \ 6, \ 7, \ 8$ we respectively have:
\begin{equation}
g_{\ell} = 7 \ , \ 5 \ , \ 4 \ , \ 3 \ , \ 2 \ , \ 2 \ , \ 2.
\end{equation}
In this paper we present the following theorem, solving the question for $\ell = 3$.
\begin{theorem} Let $r_{g, 3}: \mathcal P^{\perp}_{g, 3} \to \mathcal R_{g,3}$ be the level $3$ Mukai map then:
\begin{enumerate}
\item $r_{4,3}$ has not maximal rank,
\item $r_{5,3}$ is birational,
\item $r_{6,3}$ has not maximal rank.
\end{enumerate}
\end{theorem}
  The image of $r_{4,3}$ is contained in a divisor of $\mathcal R_{4,3}$, parametrizing pairs $(C, \eta)$ such that the multiplication map $\mu: H^0(\omega_C \otimes \eta) \otimes H^0(\omega_C \otimes \eta^{-1}) \to H^0(\omega_C^{\otimes 2})$ is not an isomorphism. This case seems interestingly related to the $G_2$-variety, see \cite{Mu} and section 7. \par
  The proof of (3) is sketched here and it will appear elsewhere. The image of $r_{6,3}$ parametrizes pairs $(C, \eta)$, where $C$ is a curvilinear section of a suitable Gushel - Mukai threefold singular along a rational normal sextic curve, see section 7. \par
Let $(S, \mathcal L, \mathcal E)$ be a level $\ell$ K3 surface of genus $g$ and $\phi: S \to \mathbb P^g$ the morphism defined by $\mathcal L$, we assume for simplicity that $\phi$ is birational onto
$\overline S := \phi(S)$. Then we close this introduction with few lines addressing the \it classification of Fano threefolds \rm  $$ \overline X \subset \mathbb P^{g+1}$$ whose general hyperplane sections are projective models  $\overline S$ as above.   The problem sounds similar to that of classifying threefolds $T \subset \mathbb P^g$ whose hyperplane sections are Enriques surfaces, that is, Enriques-Fano threefolds. It seems however quite neglected. \par  Some examples of threefolds $\overline X$ appear in this paper, most are normal and $\Sing \overline X$ is a curve. Moreover $\overline X$ admits a cyclic cover $\pi: \tilde X \to \overline X$, branched exactly on $\Sing \overline X$.  A basic notion of level $\ell$ polarized projective variety $(X, \mathcal L, \mathcal E)$ is introduced in the next section, since it is useful in the cases we want to consider. \par
 We wish to thank the referee for the careful reading and the useful advice. 
 \section*{Acknowledgement}
\it We are happy of contributing to this volume, celebrating Professor Fabrizio Catanese on the occasion of his Seventies. Let us wish to him abundance in mathematics and life as always.
\rm 
\section{Some preliminaries}
 In what follows $X$ is a smooth, irreducible complex projective variety and $\mathcal L$ is a big and nef line bundle on $X$, we say that $(X, \mathcal L)$ is a polarized projective variety. On the other hand we are  interested, along this paper, in some families of cyclic coverings 
 \begin{equation}
\pi: \tilde X \to X.
\end{equation}
Then we fix our conventions about, \cite{EV}, \cite{L} I p.242. By definition $\pi$ is a finite morphism of degree $\ell \geq 2$ and it is the quotient map of the action of an automorphism of order $\ell$ of $\tilde X$. We assume that $\tilde X$ is normal, up to  composing  $\pi$ with the normalization map.  Hence $\tilde X$ is reduced with irreducible connected components. Starting from $\pi$, we briefly  review
the recipe for its construction. Notice that $\pi_* \mathcal O_{\tilde X} \cong \mathcal A$, where
\begin{equation}
 \mathcal A = \label{mathcal A}\mathcal O_X \oplus \mathcal E^{-1} \oplus \dots \oplus \mathcal E^{-\ell+1}
 \end{equation}  and $\mathcal E \in \Pic X$. Assume $\tilde X$ is connected and hence irreducible. Then $\pi$ defines the field extension $\pi^*: k(X) \to k(\tilde X)$ and its trace map induces
 the exact sequence
 \begin{equation}
 0 \to {\mathcal E}^{- \ell} \stackrel { s} \to \mathcal O_X \to \mathcal O_B \to 0,
 \end{equation}
 for some $s \in H^0(\mathcal E^{ \ell})$. The multiplication by $s$ defines a structure of $\mathcal O_X$-Algebra on $\mathcal A$.
 We have $\tilde X = \Spec \mathcal A$, moreover $\pi$ factors through the projection 
$ u:  \mathbb P(\mathcal A) \to X$. The branch divisor of $\pi$ is ${\rm div}(s)$ and will be  denoted by  $B$. For $B$ we fix the notation
 \begin{equation}
B = m_1B_1 + \dots + m_rB_r,
\end{equation}
where $B_1, \dots, B_r$ are prime divisors.  Conversely, a pair $(\mathcal E, B)$ such that $B \in \vert \mathcal E^{\ell} \vert$ defines on $\mathcal A$ an $\mathcal O_X$- Algebra structure as above and a cyclic cover $\pi$. Notice that the condition $g.c.d. (\ell, m_1, \dots, m_r) = 1$ implies the irreducibility of $\tilde X$. \par 
Now let $C$ be a reduced curve and $\eta \in \Pic C$ a  nontrivial  $\ell$-torsion element.  Then $(C, \eta)$ uniquely defines, using a nonzero vector $s \in H^0(\eta^{\ell})$, a  nonramified  cyclic cover
$$
\pi: \tilde C \to C,
$$
which is  nontrivial.  To give a pair $(C, \pi)$ is equivalent to give a singular level $\ell$ curve $(C, \eta)$.   Now recall that a  curve $C \subset X$  is \it mobile \rm  if moves in an irreducible algebraic family covering $X$, with \it integral \rm general member. In the Neron-Severi group $N_1(X) \otimes_{\mathbb Z} \mathbb R$ the  \it mobile classes \rm of such curves generate an important convex cone,  \cite{BDPP} 1.3 (vi), \cite{L} II p. 307. Finally we introduce the following definition.
  \begin{definition}  Let $\mathcal E \in \Pic X$, the pair $(X, \mathcal E)$ is a level $\ell$ structure on $X$ if:  \par
  $\circ$ $\vert \mathcal E^{\ell} \vert \neq \emptyset$ and a general $B \in \vert \mathcal E^{\ell} \vert$ defines an integral cyclic cover, \par
  $\circ$ there exists a mobile curve $C$ in  $X$ such that $C B = 0$.  
    \end{definition}
Assume $\dim X = 1$ then $X$ is the smooth, integral curve $C$ and $\mathcal E$ is a line bundle of degree $0$ such that $\mathcal E^{\ell} \cong \mathcal O_C$. Moreover we are assuming that the cover $\pi: \tilde C \to C$
defined by $\mathcal E$ is integral. Hence $\mathcal E$ is a  nontrivial  $\ell$-torsion element. Then, for curves, the definition is the traditional one.   In higher dimension the next property is clear.
\begin{proposition} Let $(X, \mathcal E)$ be a level $\ell$ structure on $X$ and $C \subset X$ a mobile curve such that $CE = 0$, where $\mathcal O_X(E) \cong \mathcal E$.  Then 
$\mathcal O_C(E)$ is an $\ell$-torsion element of $\Pic C$.
\end{proposition}
\begin{proof}   Consider  $D \in \vert \mathcal E^{\ell} \vert$. Since $C$ is movable we can assume that $C$ is not a component of  $D$. Then $C \cap D$ is empty because $CE= 0$. This implies that $\mathcal E^{\ell} \otimes \mathcal O_C \cong \mathcal O_C(D) \cong \mathcal O_C$.  
\end{proof}
\begin{remark} \rm Nevertheless we may have a trivial $\mathcal O_C(E)$ even when $\mathcal E$ is not, and even generically when $C$ moves in its family. This is obvious if $C$ is  smooth and rational. Furthermore consider a curve $F$ and the projection $p: F \times X \to X$. Then $(F \times X, p^* \mathcal E)$ is a level $\ell$-structure on $F \times X$ and $p^* \mathcal E$ is trivial on the mobile curve $p^*(x)$, $x \in X$.
\end{remark}
Then, to address the concrete topics of our paper, we turn to polarized pairs $(X, \mathcal L)$ and we denote by $d$ the dimension of $X$. We assume that $\vert \mathcal L^{m} \vert$ is globally generated for $m >>0$ and observe that a general complete intersection of $d-1$ elements of $\vert \mathcal L^{m} \vert$ is a smooth, integral mobile curve, which
moves in an irreducible family $\mathcal C_m$ of transversal complete intersections in $X$.   \begin{proposition}  Let $X, \ \mathcal L, \ \mathcal E$ be as above. Assume $C E = 0$, where $C \in \mathcal C_m$ and $\mathcal O_X(E) \cong \mathcal E$.
Then $\mathcal O_C(E)$ is a  nontrivial  $\ell$-torsion element of $\Pic C$, moreover  $$ h^0(\mathcal O_X(kE)) = 0, \ \ k \not \equiv 0 \mod \ell. $$
 \end{proposition}
 \begin{proof} By induction on $d = \dim X$. Let $d = 1$ then $X = C$ and $\lbrace C \rbrace = \mathcal C_m$. Since $\mathcal E$ defines an integral cover, the statement follows.  Let $d \geq 2$ and  $C = D_1  \cdot\ldots \cdot  D_{d-1}$, where $D_1, \ldots, D_{d-1} \in \vert \mathcal L^{m} \vert$, then a general $D$ in the linear system generated by $D_1 \dots D_{d-1}$  is smooth. $\mathcal O_D(D)$ is nef, big and globally generated.  Let $\pi: \tilde X \to X$ be the cyclic cover, branched on $B$,  we have, since $C$ is mobile and $CB = 0$ we can assume $C \cap B = \emptyset$. Now let $f: X \to \mathbb P^n$ be the morphism defined by $\vert D \vert$, then $f$ is generically finite onto its image and the same is true for $f \circ \pi: \tilde X \to \mathbb P^n$. Then $\tilde C = \pi^{-1}(C)$ is connected by the connectedness theorem and $\mathcal O_C(E)$ is non trivial of $\ell$-torsion in $\Pic C$. Moreover $(D, \mathcal O_D(E))$ is a level $\ell$ structure and the second statement follows by induction on $d$.   \end{proof} 
Keeping this notation we finally come to the following definition.
\begin{definition}\label{DEFI} A level $\ell$ polarized variety is a triple $(X, \mathcal L, \mathcal E)$ such that $(X, \mathcal E)$ is a level $\ell$ structure on $X$ and
$CE = 0$, where $C \in \mathcal C_m$. \end{definition}
Actually the triples $(X, \mathcal L, \mathcal E)$ we will consider always satisfy the additional property: \medskip \par  \it  $\vert \mathcal L \vert$ is base point free and defines a birational morphism onto its image
\begin{equation}
f: X \to \mathbb P^n.
\end{equation}
\rm
Hence we assume $C = H_1 \cap \dots \cap H_{d-1} \in \mathcal C_1$, where $H_1 \dots H_{d-1} \in \vert f^* \mathcal O_{\mathbb P^n}(1) \vert$. So $C$ shows the distinguished line bundles  
$ \eta_C := \mathcal E \otimes \mathcal O_C$ and $\mathcal L_C :=  \mathcal L \otimes \mathcal O_C$ and these lead us to the varieties we are interested in. For these $\mathcal L_C$ is the canonical sheaf $\omega_C$. For the triples considered, we will also have that the restriction  $ r: H^0(\mathcal L) \to H^0(\omega_C)$ is surjective and that $\overline X := f(X)$ is normal.
  So we are going to deal with projective varieties $\overline X$ whose curvilinear sections are canonical curves $C$, endowed with the \'etale cover defined by $\eta_C$. This includes K3 surfaces and Fano threefolds with a prescribed level $\ell$ structure. 
 \section{ Level $\ell$ K3 surfaces} 
We begin discussing the families of level $\ell$ polarized $K3$ surfaces $(S, \mathcal L, \mathcal E)$ and the chances that $C \in  \vert \mathcal L \vert$ be  a curve with general moduli. We say that $C^2 = 2g - 2$ is the \it degree \rm of $(S, \mathcal L)$ and $g$ its \it genus\rm. As usual the moduli space of $(S, \mathcal L)$ is  denoted by 
\begin{equation}
\mathcal F_g,
\end{equation}
it is an integral quasi projective variety of dimension $19$. Let $[S, \mathcal L] \in \mathcal F_g$ be a general point, we recall that then $\Pic S \cong \mathbb Z \mathcal L$ and $\vert \mathcal L \vert$ defines an embedding
 \begin{equation}
 f: S \to \mathbb P^g
 \end{equation}
for $g \geq 3$. Coming to level $\ell$ structures $(S, \mathcal L, \mathcal E)$, these properties are no longer satisfied, as we are going to recall. We fix our notation as follows, the map
 \begin{equation}
\pi': \tilde S' \to S
\end{equation}
is the covering morphism defined by $\mathcal E$. As already established its branch divisor is
$$
B = m_1B_1 + \dots + m_rB_r,
$$
where $B_1, \dots, B_r$ are the irreducible components of $\Supp B$. Of course, since $\Pic S$ has no torsion, $B$ is not zero. We fix the following convention: 
\begin{itemize} \it
\item[$\circ$] $r$ is the number of irreducible components of $\Supp B$, 
\item[$\circ$] $t$ is the number of its  connected components.  
\end{itemize}
Moreover we set
\begin{equation}
B_1 + \dots + B_r = B_{\mathsf {red}} = N_1 + \dots + N_t,
\end{equation}
where $N_1 \dots N_t$ denote the connected components of $\Supp B$. Notice that $CB_i = 0$ for $i = 1 \dots r$. Indeed $C$ is integral and $\dim \vert C \vert \geq 1$ so that $CB_i \geq 0$. Since
$B \in \vert \ell E \vert$ then $CB = 0$ and this implies $CB_i = 0$. Then, applying the Hodge Index Theorem, $B_i$ is an integral curve on $S$ with $B^2_i < 0$. Hence $B^2_i = -2$ and $B_i$ is $\mathbb P^1$. The same argument applies to $N_j$ which is a reduced connected curve of arithmetic genus $0$. In particular each $N_j$ is contracted by $f$ to a quadratic singularity and $\Pic S$ is not isomorphic to $\mathbb Z$. \par
It is not difficult to see that the Kodaira dimension of $\tilde S'$ is zero, moreover, with some elaboration, one has the following property, cfr. \cite{G}, \cite{N1}.
\begin{proposition} Either $\tilde S'$ is birational to a K3 surface or to an abelian surface. \end{proposition}
  \begin{definition} \label{defnat} Let $(S, \mathcal L, \mathcal E)$ be a level $\ell$ K3 surface, we say that:
  \begin{enumerate}
  \item $(S, \mathcal L, \mathcal E)$ is of K3 type if $\tilde S'$ is birational to a K3 surface,
  \item $(S, \mathcal L, \mathcal E)$ is of abelian type if $\tilde S'$ is birational to an abelian surface.
  \end{enumerate}
  \end{definition}
 Case (2) is scarcely interesting for our purposes. We aim indeed  to use the curves $C \in \vert \mathcal L \vert$ in order to parametrize the moduli space $\mathcal R_{g, \ell}$ of level $\ell$ curves in low genus. 
 But in case (2) $C$ has not enough moduli for $g \geq 3$. \par
 \it We assume since now that $(S, \mathcal L, \mathcal E)$ is a level $\ell$  K3 surface of K3 type. \rm Then, to ameliorate the expositon,  we  just say with some abuse that $(S, \mathcal L, \mathcal E)$ is
 a level $\ell$ K3 surface. We say that two triples $(S_n, \mathcal L_n, \mathcal E_n)$, ($n = 1,2$), are isomorphic if there exists a biregular map $\beta:S_1 \to S_2$ such that $\beta^* \mathcal L_2 \cong \mathcal L_1$ and $\beta^* \mathcal E_2 \cong \mathcal E_1$, $i = 1,2$.\par As mentioned the classification of these triples is due to Nikulin and originates from his paper \cite{N1}. The part of interest here is the classification of pairs $(\tilde S, G)$, where $\tilde S$ is a K3 surface and $G$ is a finite group of symplectic automorphisms of $\tilde S$. There exist $14$ classes of pairs $(\tilde S, G)$ such that $G$ is commutative and $G$ is $\mathbb Z / \ell \mathbb Z$ exactly for $2 \leq \ell \leq 8$. After the classification, several papers addressed the description of the moduli and  the projective models of these K3 surfaces. It is due to mention here \cite{GP, GS1, G, GS2, vGS}. \par
 The triple  $(S, \mathcal L, \mathcal E)$ determines an associated triple $(\tilde S,\tilde{ \mathcal L}, \gamma)$, where $\gamma \in \Aut \tilde S$ is a symplectic automorphisms of order $\ell$ and $(\tilde S, \tilde{ \mathcal L})$ is a polarized K3 surface of degree $\ell(2g-2)$. We have indeed $B_{\mathsf{red}} = N_1 + \dots + N_t$, where the summands are the connected components and $-2$-curves. Let
$\nu: S \to \overline S$ be their contraction morphism, then the Cartesian square 
 \begin{equation} \label{diagram}
\begin{CD}
{\tilde S'} @>{\pi'}>> {S} \\
@V{\nu'}VV @VV{ \nu}V \\
{\tilde S} @>{\pi} >> {\overline S} \\
\end{CD}
\end{equation}
is the Stein factorization of $\nu \circ \pi'$. In it $\nu'$ is a birational morphism. Let $G \subset \Aut \tilde S'$ be the group whose quotient map is $\pi'$. As we will see $ {\pi'} ^*H^0(\mathcal L(-E))$ sits in $H^0(\tilde {\mathcal L})$ as an eigenspace of the natural representation of $G$ and defines a generator $\gamma$ of $G$. Moreover $\pi$  is the quotient map of the induced action of $G$ on $\tilde S$.  Conversely, starting from $\pi$ and the minimal desingularization $\nu$, $\pi'$ is reconstructed from  the fibre product $\pi \times_{\overline S}\nu$.
\par  In order to describe the rational singularities occurring in $\Sing \overline S$ we use the notation
\begin{equation}
\mathsf T := n_1 \mathsf T_1 + \dots + n_s \mathsf T_s, 
\end{equation}
where $\mathsf T_j$ is the singularity type and  $n_j$  the number of points of type $\mathsf T_j$  in $\Sing \overline S$.
\begin{theorem} Let $(S, \mathcal E, \mathcal L)$ be a level $\ell$ K3 surface of genus $g$, then one has $2 \leq \ell \leq 8$ and $(S, \mathcal E)$ satisfies one of the following conditions:
 \begin{enumerate} \it
 \item $\ell = 2$. One has $t = 8$, $r = 8$ \  and $\mathsf T = 8 \mathsf A_1$. \medskip \par
 \item $\ell = 3$. One has $t = 6$, $r = 12$ and $\mathsf T= 6\mathsf A_2$. \medskip \par
 \item $\ell = 4$. One has $t = 6$, $r = 14$ and $\mathsf T= 4\mathsf A_3 + 2\mathsf A_1$. \medskip \par
 \item $\ell = 5$. One has $t = 4$, $r = 16$ and $\mathsf T= 4\mathsf A_4$. \medskip \par
 \item $\ell = 6$. One has $t = 6$, $r = 16$ and $\mathsf T= 2 \mathsf A_5 + 2\mathsf A_2 + 2\mathsf A_1$. \medskip \par
 \item $\ell = 7$. One has $t = 3$, $r = 18$ and $\mathsf T= 3\mathsf A_6$. \medskip \par
 \item $\ell = 8$. One has $t = 4$, $r = 18$ and $\mathsf T = 2\mathsf A_7 + \mathsf A_3 + \mathsf A_1$.
 \end{enumerate}
  \end{theorem}
See \cite{N1}. It is also useful to observe that always one has 
\begin{equation}
E^2 = \frac {B^{  2 \color {black}}} {{\ell}^2} = - 4.
\end{equation}
Now, in view of the concrete applications in this paper,  we mention some relevant  properties of the structure of $\Pic S$ and of the moduli of the above triples.
\begin{definition} $\mathcal F_{g, \ell}$ is the moduli space of level $\ell$ K3 surfaces of genus $g$. \end{definition}
As in the case of $(S, \mathcal L)$, the construction of $\mathcal F_{g, \ell}$ relies on the usual notion of lattice polarized variety, see \cite{B, D, H} and \cite{N1} for this K3 case. In particular, for every $g \geq 2$,
$\mathcal F_{g,\ell}$ has a \it standard irreducible component \rm to be constructed as follows. We may have
\begin{equation}
\mathbb Z[\mathcal L] \oplus \mathbb M_S \subseteq \Pic S,
\end{equation}
where the sum is orthogonal. Moreover $\mathbb M_S$ has rank $r$ and it is generated by the classes $[B_1], \dots, [B_r], [E]$, with $\mathcal E \cong \mathcal O_S(E)$, so that the relation $\ell[E] - [B] = 0$ is satisfied in $\Pic S$. We can see the inclusion as the image of a primitive embedding of lattices
 \begin{equation} \label{slattice} \upsilon:  \mathbb Z c \oplus \mathbb M_{\ell} \to \Pic S, \end{equation}
where $\upsilon(c) := [\mathcal L ]$ and $\upsilon(\mathbb M_{\ell}) = \mathbb M_S$. The lattice $\mathbb M_{\ell}$ is given with the set of generators $\lbrace e , b_1,  \dots, b_r \rbrace$ so that $\upsilon(e) = [E]$, $\upsilon(b_1) = [B_1], \dots, \upsilon(b_r) = [B_r]$. Notice also that 
\begin{equation}
\label{selfint} c^2 = 2g -2 \ , \ e^2 = -4 \ , \ b_1^2 = \dots = b_r^2 = -2, 
\end{equation}
cfr. \cite{N1}. Fixing these data, the moduli space of triples $(S, \mathcal L, \mathcal E)$ endowed with an embedding $\upsilon$,  can be constructed as a moduli space of lattice polarized K3 surfaces $(S, \upsilon)$. In our case $S$ is $M$-polarized with $M := \mathbb Zc \oplus \mathbb M_{\ell}$ and the induced embedding $M \subset L := H^2(S, \mathbb Z)$ is unique up to
isometries, \cite{N1}. Then the moduli space is constructed as quotient of the period domain of these surfaces $S$.  In particular its dimension is $19 - r$\rm, \cite{D} Section 4.1 and Theorem 1.4.8, \cite{B1} Section 2.4 and Proposition 2.6. 
 Moreover a unique irreducible component of it is the closure of the moduli points of pairs $(S, \upsilon)$ such that
\begin{equation}
\Pic S = \mathbb Z[\mathcal L] \oplus \mathbb M_S.
\end{equation}
In this case we will say that $(S, \mathcal L, \mathcal E)$ is a \it standard triple \rm of genus $g$ and level $\ell$.  Let us fix our notation: 
  \begin{definition} $\mathcal F^{\perp}_{g, \ell}$ is the moduli space of standard triples of genus $g$ and level $\ell$.  \end{definition}
$\mathcal F^{\perp}_{g, \ell}$ exists for any $g \geq 2$ and $\ell = 2 \dots 8$. Fixing $\ell$, $\mathcal F^{\perp}_{g, \ell}$ is the unique irreducible component of $\mathcal F_{g, \ell}$ along a proper countable set of values $g \in \mathbb N$.  
\begin{remark} Let $(S, \mathcal L, \mathcal E)$ be a \it non standard \rm  triple and $C \in \vert \mathcal L \vert$. Then, at least experimentally for $\ell = 2$, $C$ is never general in moduli for $g \geq 4$.
This is true even when the parameter count makes that possible in low genus, see \cite{KLV2}. The situation is quite different for standard triples.  This paper studies indeed  the modular properties of $C$ in this 
case: standard behavior or peculiarities of $C$. 

\end{remark} 
  \section{A standard projective model}
  Given a standard triple $(S, \mathcal L, \mathcal E)$,  let us construct a projective realization of $S$ useful to our purposes.  
  Consider $C \in \vert \mathcal L \vert$ such that $C \cap B = \emptyset$ and  $\tilde C' = \pi'^*C$. Then the curve $\tilde C = \nu'_* \tilde C'$ is biregular to  $\tilde C'$ via the contraction $\nu': \tilde S' \to \tilde S$ and the linear map
\begin{equation}
\nu'_*: H^0(\mathcal O_{\tilde S'}(\tilde C')) \to H^0(\mathcal O_{\tilde S}(\tilde C))
\end{equation}
is an isomorphism, we identify the two spaces under it. Then, using $\tilde C$,  it is easy to remind of the action of the group $\mathbb Z / \ell \mathbb Z$ on this space and of its eigenspaces. Let
\begin{equation}
0 \to \mathcal O_{\tilde S'} \to \mathcal O_{\tilde S'}(\tilde C') \to \omega_{\tilde C} \to 0
\end{equation}
be the standard exact sequence, then $\mathbb Z / \ell \mathbb Z$ acts on its associated long exact sequence
$$
0 \to H^0(\mathcal O_{\tilde S'}) \to H^0(\mathcal O_{\tilde S'}(\tilde C')) \to H^0 (\omega_{\tilde C}) \to 0.
$$ 
As is well known the $\mathbb Z / \ell \mathbb Z$-decomposition of $H^0(\omega_{\tilde C })$ is as follows
\begin{equation}
H^0(\omega_{\tilde C }) = \bigoplus_{k = 1 \dots \ell -1} {\pi'}^* H^0(\omega_C \otimes \eta^{-k}) \bigoplus {\pi'}^* H^0(\omega_C).
\end{equation}
and this implies that $H^0(\mathcal O_{\tilde S}(\tilde C' ))$ decomposes as
\begin{equation}
H^0(\mathcal O_{\tilde S}(\tilde C' )) = \bigoplus_{k = 1 \dots \ell -1} {\pi'}^* H^0(\mathcal O_S(H_k)) \bigoplus {\pi'}^* H^0(\mathcal O_S(C)),
\end{equation}
where $\mathcal O_S(H_1) \dots \mathcal O_S(H_{\ell - 1}) \in \Pic S$ and $ \mathcal O_C(H_k) \cong \omega_C \otimes \eta^{\otimes -k}$, up to reindexing.
Since $\tilde C$ has genus $\tilde g = g + (\ell - 1)(g-1)$ it follows $\dim H^0(\mathcal O_{\tilde S}( \tilde C)) = g+1 + (\ell - 1)(g-1)$. In particular the above decomposition immediately implies that
\begin{equation}
\dim H^0( \mathcal O_S(H_k)) = \dim H^0(\omega_C \otimes \eta^{- k}) = g - 1, \ \ \ k = 1 \dots \ell - 1.
\end{equation}
In what follows, it is also useful to recall the mentioned fact that $E^ 2 = - 4$.
 \begin{lemma} It holds  $h^i(\mathcal O_S(E)) = h^i(\mathcal O_S(-E)) = 0$, for $i \geq 0$.  \end{lemma}
\begin{proof} By assumption $E$ is not effective. The same is true for $-E$, since $\ell E \sim B$ and $B > 0$. This implies $h^0(\mathcal O_S(E)) = 0$ and
$h^2(\mathcal O_S(E)) = h^0(\mathcal O_S(-E)) =  0$. Since $E^2 = -4$ we have $\chi(\mathcal O_S(E)) = 0$ and then  $h^1(\mathcal O_S(E)) = 0$. The same argument applies to $-E$. \end{proof}
Now we consider the line bundle $\mathcal O_S(C - E)$ and the standard exact sequence
$$
0 \to \mathcal O_S(-E) \to \mathcal O_S(C - E) \to \mathcal O_C(C - E) \to 0.
$$
\begin{lemma} \label{one1} Let $g \geq 2$ then the associated long exact sequence is
$$
  0 \to H^0(\mathcal O_S(C - E)) \to H^0 (\omega_C \otimes \eta^{-1} )\to 0,
$$
in particular it follows $\dim \vert C - E \vert =  g - 2$ and $h^i(\mathcal O_S(C-E)) = 0,  \ i \geq 1$.
 
\end{lemma}
\begin{proof} By the previous lemma $h^i(\mathcal O_S(E)) = h^i(\mathcal O_S(-E)) = 0$, for $i \geq 0$. Moreover we have $h^0(\omega_C \otimes \eta^{- 1}) = g- 1$ and
$h^1(\omega_C \otimes \eta^{- 1}) = 0$. Then the statement follows. \end{proof}
Now we observe that the pull-back by $\pi'$ defines a linear embedding
$$
{\pi'}^*: H^0(\mathcal O_S(C - E)) \to H^0(\mathcal O_{\tilde S'}(\tilde C')).
$$
We have indeed $\mathcal O_{\tilde S'}(\tilde C') \otimes {\pi'}^* \mathcal O_S(E - C) \cong \mathcal O_{\tilde S'}({\pi'}^*E)$ and finally
\begin{equation} \label{one2} h^0 ( \mathcal O_{\tilde S'}({\pi'}^*E)) = h^0({\pi'}_* \mathcal O_{\tilde S'}({\pi'}^*E)) = h^0(\mathcal A(E)) = 1,
\end{equation}
with $\mathcal A = \mathcal O_S \oplus \mathcal O_S(-E) \oplus \dots \oplus \mathcal O_S((1-\ell)E)$. The equality defines, up to a  nonzero  constant factor, the linear embedding ${\pi'}^*$. Then
{\rm Im} ${\pi'}^*$ is the $\mathbb Z/ \ell \mathbb Z$-invariant space $$ {\pi'}^* H^0(\mathcal O_S(C - E)). $$   
\begin{proposition} Let $g \geq 3$ and $\Pic S \cong \mathbb Zc \oplus \mathbb M_{\ell}$, then $\vert C - E \vert$ is base point free. \end{proposition}
\begin{proof} Since $S$ is a K3 surface, it suffices to prove that $\vert C - E \vert$ has no fixed component. Let $F$ be an integral fixed component of $\vert C - E \vert$, set $f = F \cdot C$ for a general $C$. Then $f$ is a fixed divisor of $\vert \omega_C \otimes \eta^{-1} \vert$. Applying Riemann-Roch
to $C$ it follows $\dim \vert \eta(f) \vert = \deg f - 1$. Since $g \geq 3$ then $deg f \leq 2$. Hence $F$ is a line, a conic or $FC = 0$. We have $F \sim xC + \sum y_jB_j +zE$ in $\Pic S$.  Assume $\deg f > 0$ then $0 < CF = (2g-2)x \leq 2$ with $x \in \mathbb Z$: a contradiction for $g \geq 3$. 
Let $CF = 0$ then $F^2 = -2$ by  the  Hodge Index Theorem and $F$ is a $\mathbb P^1$ contracted by $f_{\vert C \vert}: S \to \mathbb P^g$.
  By Lemma \ref{one1}, $h^0(C-E)=g-1=(C-E)^2/2+2$. Let $M$ be the moving part of the linear system $|C-E|$,  then $\dim \vert M \vert \geq 1$ and $MF \geq 0$.   Moreover we have
 $C - E \sim M + kF + R$, where $R$ is a curve not containing $F$ and $k \geq 1$. Let $G \in \vert M + F \vert$ be general then $G$ contains $F$: otherwise the curve $kF$ could'nt be a component of the element $G + (k-1)F + R \in \vert C - E \vert$. Hence $F$ is a fixed component of $\vert M + F \vert$. Now observe that $MF \geq 0$ and then consider the standard exact sequence
$$
0 \to \mathcal O_S(M) \to \mathcal O_S(M+F) \to \mathcal O_F(M) \to 0.
$$
We claim that, passing to the associated long exact sequence, it follows $$ \chi(\mathcal O_S(M)) = \chi (\mathcal O_S(M+F)) $$ and $\chi(\mathcal O_F(M)) = 0$. Since $F = \mathbb P^1$ this implies
$MF < 0$: a contradiction. To prove the claim consider a smooth $D \in \vert M \vert$. Then either $D$ is integral of genus $g - 2$ and $h^1(\mathcal O_S(M)) = 0$ or
$M \sim (g-2)N$ and $N$ is a smooth integral elliptic curve. Via Serre duality we have $h^2(\mathcal O_S(M)) = h^2(\mathcal O_S(M+F)) = 0$. Moreover
$MF \geq 0$ implies $h^1(\mathcal O_F(M)) = 0$. Then, in the former case,  $h^1(\mathcal O_S(M)) = 0$ implies $h^1(\mathcal O_S(M+F)) = 0$ and the claim follows. In the latter case replace $M$ by $N$. Then the equality and the same contradiction follow by the same type of arguments.  
%
\end{proof}
Now we  introduce a second linear system associated  with  $E$. At first let us set
\begin{equation}
\label{defBred} B_{\mathsf {red}} := B_1 + \dots + B_r,
\end{equation}
where the summands are the irreducible components of $\Supp B$. Then we recall that
$$
E = \frac 1{\ell} (m_1 B_1 + \dots + m_r B_r), \ \  \text {with $m_1 \dots m_r \in [1 \dots \ell-1]$.}
$$   
 \begin{definition}  Set  $ \mathring E = B_{\mathsf {red}} - E =  \frac 1{\ell} ( \mathring m_1B_1 + \dots + \mathring m_rB_r)$, where $\mathring m_i := \ell - m_i$\end{definition}
  Let us denote by $n_i$ the coefficients of the curves $B_i$ in $-\ell E$. Then  $n_i\equiv\mathring m_i\mod \ell$. More precisely, $E$ is a generator of $\mathbb{Z}/\ell\mathbb{Z}=\langle B_i, E\rangle/\langle B_i\rangle$ and $\mathring E$ is its opposite in $\mathbb{Z}/\ell\mathbb{Z}$; in particular it is a different generator of the same group. Hence  $\mathring {\mathcal E}:= \mathcal O_S(\mathring E)$ is a level $\ell$ structure, with the same properties of  $\mathcal E$ . We notice  that $\mathring E$ defines a cover $\mathring \pi': \tilde S' \to S$ so that $\mathring \pi' = \pi' \circ a$ and $a^{\ell} = id_{\tilde S'}$.  Then we define
\begin{equation}
\vert H \vert := \vert C - E \vert \ \ {\rm , } \ \ \mathring H := \vert C - \mathring E \vert.
\end{equation}
The rational maps associated  with  these linear systems respectively will be
 \begin{equation}
 p: S \to \mathbb  P \ \ , \ \  \mathring p: S \to \mathring {\mathbb  P},
\end{equation}
where $\mathbb  P := \vert  H \vert^*$ and $\mathring {\mathbb  P} := \vert \mathring H \vert^*$ are the projective space $\mathbb P^{g-2}$. Let $\iota$ be the inclusion 
\begin{equation}
\mathbb  P \times \mathring { \mathbb  P} \subset  \mathbb P^{(g-1)^2-1}
\end{equation}
defined by the Segre embedding, we set $f:= \iota \circ (p \times \ring p)$ and fix the notation
\begin{equation}
f: S \to \mathbb  P \times \mathring {\mathbb P}  \subset \mathbb P^{(g-1)^2-1}.
\end{equation}
  \begin{definition} \label{projmodel} The morphism $f$ is the main projective model of $(S, \mathcal L, \mathcal E)$. \end{definition}
 The next two remarks are simple but relevant in order to discuss $f$ (the second one follows by a direct computation of $E\cdot \mathring E$, where the class $E$ is explicitly given in \cite{N1}):
 \begin{enumerate}
\item $f^* \mathcal O_{\mathbb P^{(g-1)^2 -1}}(1) \cong \mathcal O_S(H + \mathring H) \cong \mathcal O_S(2C - B_{\sf red})$, \medskip \par
\item $ H \mathring H = 2g + 2 - t$.
\end{enumerate}
\begin{proposition}\label{H-ringH not effective}  The divisors   $[H - \mathring H]$ and $[\mathring H - H]$ are not effective classes for $\ell \geq 3$ and 
\begin{equation} \label { vanish1 }  h^1(\mathcal O_S(H - \mathring H)) = h^1(\mathcal O_S(\mathring H - H)) = 6-t.\end{equation}
\end{proposition}
\begin{proof} We have $H(H - \mathring H) = \mathring H (\mathring H - H) = t - 8$. Since the general elements of $\vert H \vert$ and $\vert \mathring H \vert$ are irreducible curves,
the first statement follows for $\ell \geq 3$ because then $t \leq 6$. The second statement just follows from Riemann-Roch.
\end{proof}
Now let us consider, for a general $C \in \vert \mathcal L \vert$,  the standard exact sequence 
\begin{equation} \label{exact}
0 \to \mathcal O_S(C - B_{\sf red}) \to \mathcal O_S(2C - B_{\sf red}) \to \mathcal O_C(2C - B_{\sf red}) \to 0.
\end{equation}
Since $C$ is smooth and disjoint from $B_{\sf red}$, then $\mathcal O_C(- B_{\sf red})$ is trivial and $\vert 2C - B_{\sf red} \vert$
cuts on $C$ a linear system of bicanonical divisors. Moreover we know that both $\vert H \vert$ and $\vert \mathring H \vert$ are base point free. Hence the same is true for
$\vert H + \mathring H \vert = \vert 2C - B_{\mathsf{red}} \vert$. Notice that
$$
(2C - B_{\mathsf{red}})^2 =  8(g-1) - 2t,
$$
which is $\geq 0$ for $g \geq 3$ and any of the prescribed values of $t, \ell$. Actually the zero value is only reached in the known situation $g = 3$, $\ell = 2$. Hence we assume $g \geq 4$ for $\ell = 2$. Then a general $D \in \vert H + \mathring H \vert$ is a smooth integral curve such that $D^2 > 0$. As is well known,  this implies $h^i(\mathcal O_S(H + \mathring H)) = 0$ for
$i \geq  1 $ and the next property follows.
\begin{proposition} Let $g$ be as above then $\dim \vert 2C - B_{\mathsf{red}} \vert = 4g - t -  3$ and the long exact sequence  associated with  the exact sequence (\ref{exact}) is as follows:
$$
0 \to H^0(\mathcal O_S(C-B_{\sf red})) \to H^0(\mathcal O_S (2C - B_{\sf red})) \to H^0(\omega_C^{\otimes 2}) \to H^1(\mathcal O_S(C-B_{\sf red}))  \to 0.
$$
\end{proposition}
The linear system $\vert C - B_{\mathsf{red}} \vert$ also deserves some observations. Since we are dealing with a general standard triple $(S, \mathcal L, \mathcal E)$, we know that $\vert C \vert$ defines a morphism
$$
f_{\vert C \vert}: S \to \mathbb P^g
$$
which is the contraction $\nu: S \to \overline S$, composed with  the embedding $\overline S \subset \mathbb P^g$  defined by $\vert \nu_* C \vert$. Since a general $C$ is disjoint
from $B$, $\vert \nu_*C \vert$ is a linear system of Cartier divisors. Let $\mathcal I_{\Sing \overline S}$ be the ideal sheaf of $\Sing \overline S$, it is clear that the natural map
$$
f^*_{\vert C \vert}: H^0(\mathcal I_{\Sing \overline S}(1)) \to H^0(\mathcal O_S(C - B_{\mathsf{red}}))
$$
is an isomorphism. Then, considering the above exact sequence (\ref{exact}),  we have 
\begin{equation}
h^0(\mathcal O_S(C - B_{\mathsf{red}})) - h^1(\mathcal O_S(C - B_{\mathsf{red}})) = \chi(\mathcal O_S(2C - B_{\mathsf{red}})) - \chi(\omega^{\otimes 2}_C) = g + 1 - t.
\end{equation}
This implies the next property.
\begin{proposition}  It holds  $h^1(\mathcal O_S(C - B_{\mathsf{red}})) = 0$ if and only if $h^0(\mathcal O_S(C - B_{\mathsf{red}})) = g + 1 - t$, that is, the points of $\Sing \overline S$ are linearly independent in $\mathbb P^g$.
\end{proposition}
On the other hand consider the commutative diagram
\begin{equation}
\begin{CD} \label{CD}
@. 0 \\
@.     @VVV \\
@. {H^0(\mathcal O_S(C - B_{\mathsf{red}})) } \\
@. @VVV \\
{H^0(\mathcal O_S( H)) \otimes H^0(\mathcal O_S(\mathring H))} @>{\mu_S}>> {H^0(\mathcal O_S( H+\mathring H))} \\
@V{\rho_H \otimes \rho_{\mathring H}}VV @V{\rho_C}VV \\
{H^0(\omega_C \otimes \eta^{-1}) \otimes H^0(\omega_C \otimes \eta)} @>{\mu_C}>> {H^0(\omega_C^{\otimes 2})} \\
@. @VVV \\
@. {H^1(\mathcal O_S(C - B_{\mathsf{red}}))} \\
@. @VVV  \\
@. 0
\end{CD}
\end{equation}
where $\mu_S$ and $\mu_C$ are the multiplication maps and the vertical arrows are the restriction maps. It follows from Lemma (\ref{one1}) that  $\rho_H \otimes \rho_{\mathring H}$ is an isomorphism.   The next property is clear.
\begin{proposition} \label{prop1} If $\mu_C$ is surjective then $h^1(\mathcal O_S(C - B_{\mathsf{red}})) = 0$ i.e. $\rho_C$ is surjective. \end{proposition}
Since $\chi (\mathcal O_S (C - B_{\mathsf{red}}) = g + 1 - t$ let us point out that $\mu_C$ is \it not surjective \rm if  
\begin{equation}
g < t - 1.
\end{equation}
 We do not further investigate the diagram, for our applications these results suffice.
   \section{Views on the Mukai maps in level $\ell$}
  In this section we only put in large the picture we have outlined in the introduction. This picture concerns the maps in (\ref{mukai}) and (\ref{lmukai}), that is, the Mukai map
$$
m_g: \mathcal P_g \to \mathcal M_g
$$
and the level $\ell$ Mukai maps
$$
r_{g,\ell}: \mathcal P^{\perp}_{g,\ell} \to \mathcal R_{g, \ell}.
$$
These maps, and the involved moduli spaces, have been previously considered. We recall that the points of $\mathcal P_g$ are the elements $[S, \mathcal L, C]$ such that $[S, \mathcal L] \in \mathcal F_g$ and $C \in \vert \mathcal L \vert$. The Mukai map
$m_g$ is the natural forgetful map. We have
 \begin{enumerate} \it
\item $m_g$ is dominant for $g \leq 9$,  
\item $m_g$ is not dominant for $g = 10$,
\item $m_g$ is birational for $g = 11$,
\item $m_g$ has $1$-dimensional fibre for $g = 12$.
\item $m_g$ is generically injective for $g \geq 13$.
\end{enumerate}
Thus $m_g$ has not maximal rank for $g = 10, 12$. It is indeed known that a general $[C] \in m_{10}(\mathcal P_{10})$ is a linear section $C$ of the $G_2$ variety $W \subset \mathbb P^{13}$, \cite{Mu}. Hence the family of $2$-dimensional linear sections of $W$ through $C$ is a $\mathbb P^3$. It turns out from this fact that the fibre of $m_{10}$ at $[C]$ is $3$-dimensional. Then
$m_{10}(\mathcal P_{10})$ has codimension $1$. Genus $12$ Fano threefolds play a similar role, then a general fibre of $m_{12}$ is a rational curve. \par  In this perspective, asking about the connections between the moduli space $\mathcal F^{\perp}_{g, \ell}$, of level $\ell$ K3 surfaces of genus $g$, and $\mathcal R_{g, \ell}$ is, as observed,  natural. For a general point  $[S, \mathcal L, \mathcal E] \in \mathcal F^{\perp}_{g, \ell}$ one can ask if $(C, \eta)$, with $C \in \vert \mathcal L \vert$ and $\eta = \mathcal E \otimes \mathcal O_C$, defines a general point of $\mathcal R_{g, \ell}$. More precisely recall that 
  $\mathcal P^{\perp}_{g, \ell}$  is  the moduli space of $4$-tuples $(S, \mathcal L, \mathcal E, C)$ such that  $[S, \mathcal L, \mathcal E] \in \mathcal F^{\perp}_{g,\ell}$ and $C \in \vert \mathcal L \vert$.  
 The level $\ell$ Mukai $ r_{g,\ell}: \mathcal P^{\perp}_{g, \ell} \to \mathcal R_{g,\ell}$ is the morphism sending $[S, \mathcal L, \mathcal E, C] \in \mathcal P^{\perp}_{g, \ell}$  to the point $[C, \eta_C] \in \mathcal R_{g, \ell}$, where $\eta$ is $\mathcal E \otimes \mathcal O_C$. About the possible dominance of the  map $r_{g,\ell}$ we have:   
 \begin{enumerate} \it
 \item $3g-3 = \dim \mathcal R_{g,2} \leq \dim \mathcal P^{\perp}_{g,  2 } = 11 + g$ iff $g \leq 7$. \medskip
 \item $3g-3 = \dim \mathcal R_{g,3} \leq \dim \mathcal P^{\perp}_{g,  3 } = \ 7 + g$ iff $g \leq 5$. \medskip
 \item $3g-3 = \dim \mathcal R_{g,4} \leq \dim \mathcal P^{\perp}_{g,  4 } = \ 5 + g $ iff $g \leq 4$. \medskip
 \item $3g-3 = \dim \mathcal R_{g,5} \leq \dim \mathcal P^{\perp}_{g, \color {blue} 5 } = \ 3 + g$ iff $g \leq 3$. \medskip
 \item $3g-3 = \dim \mathcal R_{g,6} \leq \dim \mathcal P^{\perp}_{g,  6 } = $  \  $3 + g$ iff $g \leq 3$.  \medskip
 \item $3g-3 = \dim \mathcal R_{g,7} \leq \dim \mathcal P^{\perp}_{g,  7 } = \ 1 + g$ iff $g \leq 2$. \medskip
 \item $3g-3 = \dim \mathcal R_{g,8} \leq \dim \mathcal P^{\perp}_{g,  8 } = \ 1 + g$ iff $g \leq 2$. \medskip
  \end{enumerate}
These issues have not been systematically considered but for $\ell = 2$.   We close this expository section with a summary on what happens for $\ell = 2, 3$.
\subsection { \sf The picture for $\ell = 2$} 
 We have $3g-3 = \dim \mathcal M_g \leq \dim \mathcal P^{\perp}_{g,2} = 11 + g$ iff $g \leq 7$. Again, $r_{g,2}$ behaves unexpectedly near the value of transition, which is now $g = 7$.
  \begin{enumerate} \it
\item $r_{g,2}$ is dominant for $g \leq 5$,  \medskip
\item $r_{g,2}$ is not dominant for $g = 6$, \medskip
\item $r_{g,2}$ is birational for $g = 7$, \medskip
\item $r_{g,2}$ has not finite fibres for $g = 8$. \medskip
\item $r_{g,2}$ is generically injective for $g \geq 9$.
\end{enumerate}
These surfaces are known as (standard) \it Nikulin surfaces. \rm Cases (1), (2), (3) are treated in \cite{FV, FV1},  the remaining ones, (standard and non standard), in \cite{ KLV1, KLV2}. Notice that \it $r_{g,2}$ is not of maximal rank for $g = 6,8$. \rm In genus $6$  the condition $C \subset S$ implies that the following multiplication map is not an isomorphism as expected:
\begin{equation}
\mu: {\rm Sym}^2 H^0(\omega_C \otimes \eta_C) \to H^0(\omega_C^{\otimes 2}).
\end{equation}
Then $(C, \eta_C)$ does not define a general point of $\mathcal R_{g,2}$, see \cite{B}. We point out that, studying the two cases where $r_{g,2}$ has not maximal rank, two families of singular Fano threefolds appear. Their hyperplane sections are singular models $\overline S$ of general Nikulin surfaces $S$. The existence of these threefolds implies the failure of the maximal rank.  \medskip   \par \subsection{ \sf The picture for $\ell = 3$}
 We will prove that  $r_{g,3}$ behaves unexpectedly near $g = 5$:
  \begin{enumerate}  \it
\item $r^s_{g,3}$ is dominant for $g \leq 3$,  \medskip
\item $r^s_{g,3}$ has not maximal rank for $g = 4$, \medskip
\item $r^s_{g,3}$ is birational for $g = 5$, \medskip 
\item $r^s_{g,3}$ has not maximal rank for $g = 6$. \medskip
\end{enumerate}
\begin{remark} \rm  The case $g \geq 7$ should be considered for further investigation, addressing the generic injectivity. The (uni)rationality of $\mathcal R_{g,3}$ is known,  or elementary ,  for $g \leq 5$,   cfr. \cite{BaC, BaV, Ve1}.  We recall that $\mathcal R_{g,3}$ is of general type for $g \geq 12$ and of Kodaira dimension $\geq 19$ for $g = 11$,  \cite{CEFS}. Bruns proved in \cite{Br} that $\mathcal R_{8,3}$ is of general type.   The cases $g = 6, 7, 9, 10$ and partially $g =11$ are open. \end{remark}

  \section{The Mukai map in level $3$}
\medskip \par \subsection { \sf The case of genus $4$}
Let $[S, \mathcal L, \mathcal E, C] \in \mathcal P^{\perp}_{g, \ell}$ be general and $\ell = 3$, as in section 2, (\ref{CD}) we consider the commutative diagram   
 \begin{equation}
\begin{CD} 
{H^0(\mathcal O_S( H)) \otimes H^0(\mathcal O_S(\mathring H))} @>{\mu_S}>> {H^0(\mathcal O_S( H+\mathring H))} \\
@V{\rho_H \otimes \rho_{\mathring H}}VV @V{\rho_C}VV \\
{H^0(\omega_C \otimes \eta^{-1}) \otimes H^0(\omega_C \otimes \eta)} @>{\mu_C}>> {H^0(\omega_C^{\otimes 2})}. \\
\end{CD}
\end{equation}
 Since $\ell = 3$ we have $t = 6$ connected components of $\Supp B$. Then, by proposition (\ref{prop1}),  $\mu_C$ is \it not 
surjective \rm if $g < t - 1 = 5$. This is obvious for $g \leq 3$.   For $g = 4$ the dimension count suggests that  in $\mathcal R_{4,3}$ the map $\mu_C$ is not surjective in codimension $1$ .    \begin{proposition} Let $[C, \eta] \in \mathcal R_{4,3}$ be a general point then $\mu_C$ is surjective, moreover the locus of points such that $\mu_C$ is not surjective is an effective Cartier divisor
 in $\mathcal R_{4,3}$.
  \end{proposition}
   Indeed, for $g = 4$ and $\ell = 3$, this locus turns out to be the locus $\mathcal D_{g,\ell}$ defined in \cite{CEFS} p. 77. There, for low level $\ell \geq 3$ and for $g \leq16$,  the so defined \it Torsion bundle conjecture B \rm is proven, which implies that $\mathcal D_{4,3}$ is an effective Cartier divisor in $\mathcal R_{4,3}$.  Then the next theorem follows. Notice also that, for $g = 4$,  theorem 1.7 of \cite{BaV} implies that $\mu_C$ is an isomorphism for a general $(C, \eta)$.   
\begin{theorem} The map $r_{4,3}: \mathcal P^{\perp}_{4,3} \to \mathcal R_{4, 3}$ fails to be dominant. \end{theorem}
\begin{remark} \rm The case $g = 4$ turns out to be of special interest. See the last section for a natural, presently conjectural, geometric interpretation. \end{remark}

\subsection {\sf The case of genus 5}
Differently from the case $g \leq 4$ the multiplication map
$$
\mu_C: H^0(\omega_C \otimes \eta) \otimes H^0(\omega_C \otimes \eta^{-1}) \to H^0(\omega_C ^{\otimes 2})
$$
can be surjective for $g \geq 5$ and a general point $[C, \eta] \in \mathcal R_{g,3}$. This property occurs in genus $g = 5$ and makes possible the proof of the next \it birationality theorem. \rm
\begin{theorem} \label{bir} The Mukai map $r_{5,3}: \mathcal P^{\perp}_{5,3} \to \mathcal R_{5,3}$ is birational. \end{theorem}
 Before proving  it we cannot avoid a long series of preliminaries. We will always assume that $ [S, \mathcal L, \mathcal E, C] \in \mathcal P^{\perp}_{5,3} $ is a  \it general point\rm, in particular  $\Pic S \cong \mathbb Z c \oplus \mathbb M_3$. Let
\begin{equation}
0 \to \mathcal O_S(H + \mathring H - C) \to \mathcal O_S(H + \mathring H) \to \omega_C^{\otimes 2} \to 0.
\end{equation}
be the standard exact sequence, at first we point out the following fact.  
\begin{proposition} The associated long exact sequence is
\begin{equation} \label{iso}
0 \to H^0(\mathcal O_S(H + \mathring H)) \stackrel {\rho_C} \to H^0(\omega_C^{\otimes 2}) \to 0.
\end{equation}
\end{proposition}
Since $H + \mathring H - C \sim C - B_{\mathsf{red}}$, the next lemma implies the previous statement.
\begin{lemma}\label{initial}  It holds   $h^i(\mathcal O_S(C - B_{\mathsf{red}})) = 0$ for $i \geq 0$. \end{lemma}
 \begin{proof}  Since  $C(B_{\mathsf{red}} - C) < 0$, $h^0(\mathcal O_S(B_{\mathsf{red}} - C)) = 0$. Hence $h^2(\mathcal O_S(C - B_{\mathsf{red}}))$ is zero by Serre duality.
 Since $ (C- B_{\mathsf{red}})^2 = - 4$ then $ \chi(\mathcal O_S(C - B_{\mathsf{red}})) = 0$ and the statement follows if $h^0(\mathcal O_S(C - B_{\mathsf{red}})) = 0$. Assume $A \in \vert C - B_{\mathsf{red}} \vert$  then $A$ is not connected. This follows from $ \chi(\mathcal O_S(A)) = h^0(\mathcal O_S(A)) - h^1(\mathcal O_S(A)) = 0$ and the standard exact sequence
 $$
 0 \to \mathcal O_S(-A) \to \mathcal O_S \to \mathcal O_A \to 0.
 $$ 
 This implies $A = A_1 + A_2$,
where $A_1$ is a connected component and $A_2 = A - A_1$ is a curve. We have $C(A_1+ A_2) = C(C - B_{\mathsf{red}})= 8$ and we can choose $A_1$ so that $CA_1 > 0$.
Assume $CA_2 = 0$ then the morphism $\phi: S \to \mathbb P^5$, defined by $\vert C \vert$, maps
birationally $A_1 + A_2 + B_{\mathsf{red}}$ onto a degree $8$ hyperplane section of $\overline S = \phi(S)$. This is the curve $\phi_*A_1$,  singular at the points of $\phi (B_{\mathsf{red}}) = \Sing \overline S$. 
 These points are the images by $\phi$ of the six connected components of $B_{red}$ and are exactly six. Indeed each fibre of $\phi$ is connected and hence two connected components $V_1, V_2$
of $B_{red}$, contracted to the same point, are connected by an effective divisor $W$ orthogonal to $C$. On the other hand, under our generality assumption, we have $\Pic S \cong \mathbb Zc \oplus \mathbb M_3$. Moreover a direct computation shows that, in the negative definite lattice $\mathbb M_3$, $\Supp W$ is union of irreducible components of $B_{red}$. Actually one computes that the only classes of irreducible $(-2)$-curves are the classes of $B_1 \dots B_{12}$. This implies $W = 0$ and $V_1 = V_2$.  But then $\phi_*A_1$ is not integral, because it is a hyperplane section of $\phi(S)$ with six singular points. Then  there exists an irreducible component $R$ of it such that $0 < CR < 8$. The same is obvious if $CA_2 > 0$. Since $\Pic S \cong \mathbb Zc \oplus \mathbb M_3$ we have  $[R] =x[C] + \sum y_i [B_i] + z[E]$, with $x, y_i, z \in \mathbb Z$. But this implies $0 < CR = x8 < 8$ with $x \notin \mathbb Z$: a contradiction. \end{proof}
\begin{proposition} The linear systems $\vert H  \vert$ and $\vert  \mathring H  \vert$ are not hyperelliptic. \end{proposition}
\begin{proof} Let $\vert H \vert$ be hyperelliptic, then $\vert H \vert$ defines a $2:1$  morphism $\psi: S \to \mathbb P^3$ onto a quadric surface $Q := \psi(S)$. As is well known  the pull-back of a ruling of lines of $Q$ defines a pencil $\vert F_2 \vert$ of curves such that $F_2^2 = 0$ and $HF_2 = 2$. Moreover $\vert F_1 \vert := \vert H - F_2 \vert$ is a pencil of irreducible elliptic curves. The same is true for the moving part of $\vert F_2 \vert$. Since $H \sim F_1 + F_2$ and $C \sim H + E$ we have $C (F_1 + F_2) = 8$ and also $CF_i \geq 2$, $i = 1,2$. Let 
$\vert F \vert$ be the moving part of the pencil $\vert F_i \vert$ such that $CF_i$ is minimal, then it  follows $2 \leq CF \leq 4$. On the other hand we have $F \sim xC + \sum y_j B_j + zE$ in $\Pic S$. This implies $2 \leq CF = 8x \leq 4$ and $x \notin \mathbb Z$: a contradiction. The same argument works for $\vert \mathring H \vert$. \end{proof}
\begin{lemma} It holds  $h^i(\mathcal O_S(2H - \mathring H))   = h^i(\mathcal O_S(2 \mathring H - H))  = 0$ for $i \geq 0$. \end{lemma}  
\begin{proof} From $H \sim C - E$ and $\mathring H \sim C - \mathring E$ we have $2H - \mathring H \sim C - 2E + \mathring E$, moreover $$\mathring H (\mathring H - 2H) = -8 \Rightarrow h^0 (\mathcal O_S(\mathring H -2H)) = 0 \Rightarrow h^2(\mathcal O_S(2H - \mathring H)) = 0. $$ Since $(2H - \mathring H)^2 = -4$ then $\chi(\mathcal O_S(2H - \mathring H)) = 0$. Hence the statement follows
for $2H - \mathring H $ if we prove $h^0(\mathcal O_S(2H - \mathring H)) = 0$. For this we observe that the well known descriptions of $E$ and $\mathring E$ are as follows. For $i = 1 \dots 6$ consider $N_i = B_i + B'_{i}$, that is,  the $i$-th connected component of $B_{\mathsf{red}} = \sum_{i = 1 \dots 6} B_i +  B'_{i} $. Then in $\Pic S$ we have
\begin{equation}
[E] = \sum_{i = 1 \dots 6} \frac 13 [B_i + 2B'_i] \ , \ [\mathring E] = \sum_{i = 1 \dots 6} \frac 13 [2B_i + B'_i]
\end{equation}
up to exchanging $E$ with $\mathring E$. Since $2H - \mathring H \sim C - 2E + \mathring E$, it follows that
\begin{equation} \label{noteff}
2H - \mathring H \sim C - \sum_{i = 1 \dots 6}  B'_i.
\end{equation}
This implies that $[2H - \mathring H]$ is not an effective class. Indeed let $B' := B'_1 + \dots + B'_6$, observe that $(C - B')B_i = - 1$, $i = 1 \dots 6$. Assume $C - B' \sim F$ where $F$ is  an effective divisor. Then $FB_i = -1$ implies $B_i \subset F$ and $F = F' + B_1 + \dots + B_6$ where $F'$ is effective. Hence $C - B_{\mathsf{red}}\sim F' > 0$: a contradiction to the above lemma (\ref{initial}). 
 \end{proof}
We will profit of genus $3$ curves of the non hyperelliptic linear systems $\vert H \vert$ or $\vert \mathring H \vert$. 
\begin{lemma}  It holds  $\forall \ D \in \vert H \vert, \ h^0(\mathcal O_D(\mathring H - H)) = 0$ \ and \ $\forall \ \mathring D \in \vert \mathring H \vert, \  h^0(\mathcal O_{\mathring D}(H - \mathring H)) = 0$. \end{lemma}
\begin{proof} Let $D \in \vert H \vert$, once more consider the standard exact sequence 
$$
0 \to \mathcal O_S(\mathring H - 2H) \to \mathcal O_S (\mathring H - H) \to \mathcal O_D(\mathring H - H) \to 0
$$
and its long exact sequence. We have $h^1(\mathcal O_S(\mathring H - 2H)) = h^1(\mathcal O_S(2H - \mathring H)) = 0$ by the previous lemma and $h^0(\mathcal O_S(\mathring H - 2H)) = 0$
because $H(\mathring H - 2H) = -2$. Then it follows $h^0(\mathcal O_D(\mathring H - H)) = h^0(\mathcal O_S(\mathring H - H))$. Finally the latter is zero by Proposition (\ref{H-ringH not effective}).
\end{proof}
Let $D \in \vert H \vert$ be \it smooth \rm  then $\mathcal O_D(\mathring H - H) \cong \mathcal O_D(b)$, where $\deg b = 2$. We fix the notation $b$ for such a divisor and the notation $\mu_D$ for the 
following multiplication map:
 \begin{equation}
\label{NOTEFF} \mu_D: H^0(\omega_D) \otimes H^0(\omega_D(b)) \to H^0(\omega^{\otimes 2}_D(b)).
\end{equation}
 Let us also point out that $h^0(\mathcal O_D(b)) = 0$ by the above lemma. Moreover we fix the notation
 \begin{equation}
\nu_D: H^0(\mathcal O_S(H)) \to H^0(\omega_D) \ , \ \mathring \nu_D: H^0(\mathcal O_S(\mathring H)) \to H^0(\omega_D(b)) \ , \ \rho_D: H^0(\mathcal O_S(H + \ring H)) \to H^0(\omega_D^{\otimes 2}(b))
\end{equation}
for the natural restriction maps.  Then  we consider the commutative diagram:
\begin{equation}
\begin{CD} \label{CD five}
{H^0(\mathcal O_S( H)) \otimes H^0(\mathcal O_S(\mathring H))} @>{\mu_S}>> {H^0(\mathcal O_S( H+\mathring H))} \\
@V{\nu_D \otimes \mathring \nu_D}VV @V{\rho_D}VV \\
{H^0(\omega_D) \otimes H^0(\omega_D(b))} @>{\mu_D}>> {H^0(\omega_D^{\otimes 2}(b))}. \\
\end{CD}
\end{equation}
which is similar to our main diagram (\ref {CD})
 \begin{proposition} The vertical arrows and the horizontal arrow $\mu_D$ are surjective.  \end{proposition}
\begin{proof} Let $p: S \to \mathbb P^3$ be the map defined by $\vert H \vert$, then $p \vert D: D \to \mathbb P^2 = \vert \omega_D \vert^*$
is the canonical map and $\vert \omega_D(b) \vert$  is cut on $D$ by $\vert \mathcal I_{d \vert S}(3H) \vert$, where $d$ is any element of $\vert \omega^{\otimes 2}(-b)\vert$ and $\mathcal I_{d \vert S}$
 is its ideal sheaf. Moreover the map $p^*: \vert \mathcal O_{\mathbb P^2}(3) \vert \to \vert \omega_D^{\otimes 3} \vert$ is an isomorphism and $\vert \mathcal I_{d \vert S}(3H) \vert$ $=$ $p^* \vert \mathcal I_{Z \vert \mathbb P^2}(3) \vert$, where $Z = p_*d$ and $\mathcal I_{Z \vert \mathbb P^2}$ is its ideal sheaf.  Hence it follows $ h^0(\mathcal I_{Z \vert \mathbb P^2}(2)) = h^0(\omega_D^{\otimes 2}(- b)) = h^0(\mathcal O_D(b)) = 0$ and $h^1(\mathcal O_D(b)) = h^0(\mathcal O_D(b)) = 0$.  This easily implies $h^i(\mathcal I_{Z \vert \mathbb P^2}(3-i)) = 0$ for $i > 0$, that is, $\mathcal I_{Z \vert \mathbb P^2}$ is $3$-regular. Hence, by Castelnuovo-Mumford regularity theorem,  the multiplication map
\begin{equation}
 \mu: H^0(\mathcal O_{\mathbb P^2}(1)) \otimes H^0(\mathcal I_{Z \vert \mathbb P^2}(3)) \to H^0(\mathcal I_{Z \vert \mathbb P^2}(4)))
\end{equation}
is surjective. Now consider the standard exact sequence of ideal sheaves
$$
0 \to \mathcal I_{p(D) \vert \mathbb P^2}(4) \to \mathcal I_{Z \vert \mathbb P^2}(4) \stackrel{\rho} \to \mathcal I_{Z \vert p(D)}(4) \to 0
$$
and its associated long exact sequence. Since $\mathcal I_{p(D) \vert \mathbb P^2}(4) \cong \mathcal O_{\mathbb P^2}$ it follows that 
$$ h^0(\rho): H^0(\mathcal I_{Z \vert \mathbb P^2}(4)) \to H^0(\omega^{\otimes 2}_D(b)) $$ is surjective.
On the other hand we have $\mu_D \circ \lambda = h^0(\rho) \circ \mu$, where $\lambda$ is the tensor product $$\lambda_1 \otimes \lambda_2: H^0(\mathcal O_{\mathbb P^2}(1)) \otimes H^0(\mathcal I_{Z \vert \mathbb P^2}(3)) \to H^0(\omega_D) \otimes H^0(\omega_D(b))$$
of the natural  isomorphisms $\lambda_1: H^0(\mathcal O_{\mathbb P^2}(1)) \to H^0(\omega_D)$ and $\lambda_2: H^0(\mathcal I_{Z \vert \mathbb P^2}(3)) \to H^0(\omega_D(b))$. Since $\lambda$ is an isomorphism and $h^0(\rho)$ and $\mu$ are surjective, then $\mu_D$ is surjective.  The surjectivity of $\rho_D$ follows from the vanishing of $h^1(\mathcal O_S(\mathring H))$ and the standard exact sequence
$$
0 \to \mathcal O_S(\mathring H) \to \mathcal O_S(H + \mathring H) \to \omega_D^{\otimes 2}(b) \to 0.
$$
Since $\omega_D^{\otimes 2}(b)$ is $\mathcal O_D(H + \mathring H)$, the surjectivity of $\nu_D$ follows from the above exact sequence twisted by  $-\mathring H$. Finally the exact sequence
$$
0 \to \mathcal O_S(\mathring H - H) \to \mathcal O_S(\mathring H) \to \omega_D(b) \to 0
$$
implies that $\mathring \nu_D$ is an isomorphism. Indeed we have $h^0(\mathcal O_S(\mathring H - H )) = h^1(\mathcal O_S(\mathring H - H))$ $ = 0$ in its long exact sequence by  (\ref{ vanish1 }).  Hence $\nu_D \otimes \mathring \nu_D$ is surjective too.   \end{proof}
\begin{proposition}  The map \label{proof} $\mu_S: H^0(\mathcal O_S( H)) \otimes H^0(\mathcal O_S(\mathring H)) \to H^0(\mathcal O_S( H+\mathring H))$ is surjective. \end{proposition}
\begin{proof} Let us consider again the commutative diagram (\ref{CD five}), that is,
$$ 
\begin{CD}  
{H^0(\mathcal O_S( H)) \otimes H^0(\mathcal O_S(\mathring H))} @>{\mu_S}>> {H^0(\mathcal O_S( H+\mathring H))} \\
@V{\nu_D \otimes \mathring \nu_D}VV @V{\rho_D}VV \\
{H^0(\omega_D) \otimes H^0(\omega_D(b))} @>{\mu_D}>> {H^0(\omega_D^{\otimes 2}(b))}. \\
\end{CD}
$$
Counting dimensions we have  $\dim \Ker \mu_S \geq  4$, hence it suffices to show that the equality holds. Now we know that $\mu_D$ and $\nu_D \otimes \mathring \nu_{D}$ are surjective.
Let $\mathbb K$ be the Kernel of $\mu_D \circ (\nu_D \otimes \mathring \nu_D)$, then the dimension count gives $\dim \mathbb K = 8$ and, of course, we have $\Ker \mu_S \subseteq \mathbb K$.
Therefore, to prove $\dim \Ker \mu_S = 4$, it suffices to produce a $4$-dimensional subspace $V \subset \mathbb K$ such that $V \cap \Ker \mu_S = (0)$. To this purpose consider the
space of decomposable vectors  $V := \langle s \rangle \otimes H^0(\mathcal O_S(\mathring H))$, where $s$ is  nonzero  and ${ \rm div } (s) = D$. Then we have $(\nu_D \otimes \mathring \nu_D) (V) =(0)$
and hence $V \subset \mathbb K$. On the other hand let $t \in H^0(\mathcal O_S(\mathring H))$, then $\mu_S(s \otimes t) = st$ and this is zero iff $t = 0$. Hence $V \cap \Ker \ \mu_S = (0)$. \end{proof} 
Now we go back, in genus $5$, to our usual diagram (\ref {CD}) in section 2. This is
\begin{equation}
 \label{NEW}
\begin{CD} 
{H^0(\mathcal O_S( H)) \otimes H^0(\mathcal O_S(\mathring H))} @>{\mu_S}>> {H^0(\mathcal O_S( H+\mathring H))} \\
@V{\rho_H \otimes \rho_{\mathring H}}VV @V{\rho_C}VV \\
{H^0(\omega_C \otimes \eta) \otimes H^0(\omega_C \otimes \eta^{-1})} @>{\mu_C}>> {H^0(\omega_C^{\otimes 2}).} \\
\end{CD}
\end{equation}
 \begin{proposition} $\mu_C: H^0(\omega_C \otimes \eta) \otimes H^0(\omega_C \otimes \eta^{-1}) \to H^0(\omega_C^{\otimes 2})$ is surjective. \end{proposition}
\begin{proof} We have already shown that $\mu_S$ and $\rho_H \otimes \rho_{\mathring H}$ are surjective. By (\ref{iso}) and its related lemma the same is true for $\rho_C$. 
Hence the surjectivity of $\mu_C$ follows. \end{proof}
Let $\mathbb P^{15}: = \mathbb P( H^0(\mathcal O_S( H))^* \otimes H^0(\mathcal O_S(\mathring H))^*)$ and let $\mathbb P^3 \times \mathbb P^3 := \iota( \vert H \vert^* \times \vert \mathring H \vert^*)$ 
be the image in $\mathbb P^{15}$ of the Segre embedding $\iota$. Now we study the morphism defined in (\ref{projmodel})
$$
f: S \to \mathbb P^3 \times \mathbb P^3 \subset \mathbb P^{15},
$$
that is, $f = \iota \circ (p \times \mathring p)$. Since the map $\mu_S$ is surjective it follows that
 \begin{equation}
 (p \times \mathring p)^* H^0(\mathcal O_{\mathbb P^3 \times \mathbb P^3}(1,1)) = H^0(\mathcal O_S(H + \mathring H)).
\end{equation}
Let $\mathbb P^{11} \subset \mathbb P^{15}$ be the linear embedding of $\mathbb P( {\rm Im \ \mu_S^*})$ defined by $\mu^*_S$, then we have
\begin{equation}
f(S) \subseteq \mathbb P^{11} \cdot (\mathbb P^3 \times \mathbb P^3) \subset \mathbb P^{15},
\end{equation}
In other words $f$ is just the morphism defined by the complete linear system $\vert H + \ring H \vert$ composed with the linear embedding $\mathbb P^{11} \subset \mathbb P^{15}$.

 \begin{proposition}  The map $p \times \mathring p$ is an embedding for a general point $[S, \mathcal L, \mathcal E] \in \mathcal F^{\perp}_{5,3}$. \end{proposition}
\begin{proof} The linear systems $\vert H \vert$ and $\vert \mathring H \vert$   are non hyperelliptic. Hence $p$, $\mathring p$ are generically
injective and the same is true for $f$. In particular $f: S \to f(S)$ is biregular over $f(S) - \Sing f(S)$ and $\Sing f(S)$ is a finite set of rational double points.  
Let $R \subset S$ be an integral curve contracted by $f$ then $R$ is biregular to $\mathbb P^1$ but it is not $B_i$. Indeed $R$ is contracted by $p$ and $\ring p$ while $B_i$ is not, as one can directly compute.
 Notice also that $C \sim \frac 12(H + \ring H + B_{red})$. Therefore, since $RC \geq 0$, it follows  $$ RC = \frac 12 \sum_{i = 1 \dots 12} RB_i \geq 0$$ with $RB_i \geq 0$.
Assume $RB_i = 0$ for each $i$,  then $RC=0$. Since the Picard group of $S$ is $\mathbb Z[\mathcal{L}]\oplus \mathbb{M}_3$, $R$ is necessarily contained in $\mathbb{M}_3=\mathbb{Z}[\mathcal{L}]^{\perp}$. By \cite{G} the unique $(-2)$-curves contained in $\mathbb{M}_3$ are the $B_i$'s, which contradicts the fact that $R$ cannot be a $B_i$..
Now assume that $RB_i \geq 2$ for some $B_i$ and consider, among the maps $p$ and $\mathring p$, the one not contracting $B_i$, say $p$. Then $p$ embeds $B_i$ as a line. On the other hand $p$
contracts $R \cdot B_i$, which is a divisor of degree $\geq 2$ in $B_i$: a contradiction. This implies $RB_i = 1$ for each $i$. Finally consider two distinct curves as above, say $B_1$ and $B_2$, which are 
contracted by $p$. Let us also claim that $p(B_1)$ and $p(B_2)$ are distinct points for a general $(S, \mathcal L, \mathcal E)$. Since $RB_1 = RB_2 = 1$ then $p(R)$ is not a point: a contradiction. \par
 We now prove that $p(B_1)\neq p(B_2)$ for a general $(S, \mathcal L, \mathcal E)$. If two curves are contracted by a map $p$ to the same point, there is a tree of $(-2)$-curves connecting these curves which is contracted by $p$. Since $p$ is defined by $|H|$, the $(-2)$-curves contracted by $p$ are orthogonal to $H$ in $\mathbb Z[\mathcal{L}]\oplus \mathbb{M}_3$, which is the Picard group of a general $S$. By a direct computation one observes that the negative defined lattice orthogonal to $H$ contains exactly 12 $(-2)$-classes, which are $\pm B_i$ for $i=1,\ldots ,6$. Since $B_iB_j=0$ if $i,j\in\{1,\ldots,6\}$ and $i\neq j$, $p(B_1)\neq p(B_2)$.
\end{proof}
 At this point the special geometry determined by $\mu_S$ appears, we have 
\begin{equation}
\Ker \mu_S = H^0(\mathcal I(1,1)),
\end{equation}
where $\mathcal I$ is the ideal sheaf of $\mathbb P^{11} \cdot (\mathbb P^3 \times \mathbb P^3)$ in $\mathbb P^3 \times \mathbb P^3$ and  $\dim \Ker \mu_S = 4$. Let
\begin{equation}
\Sigma := \mathbb P^{11} \cdot (\mathbb P^3 \times \mathbb P^3),
\end{equation}
then $f(S)$ sits in $\mathbb P^{11}$ as a K3 surface of degree $20$ and $f(S) \subseteq \Sigma$. Now assume that the intersection scheme $\Sigma$ is proper, then
$\Sigma$ is a K3 surface of degree $20$ and hence
 \begin{equation}
f(S) = \Sigma.
\end{equation}
Postponing its proof, we therefore assume the following claim. \medskip \par
{\sc claim} \it For a general triple $(S, \mathcal L, \mathcal E)$ the intersection scheme $\Sigma$ is proper. \rm \medskip \par
Then we prove the \it birationality of the Mukai map  $ r_{5,3}: \mathcal P^{\perp}_{5,3} \to \mathcal R_{5,3}$. \rm 
\begin{proof}[Proof of the birationality]  Since $\mathcal P^{\perp}_{5,3}$ and $\mathcal R_{5,3}$ are irreducible of the same dimension, it suffices to show that $r_{5,3}$ is birational onto 
$\mathcal M := r_{5,3}(\mathcal P^{\perp}_{5,3})$. Let $x = [S, \mathcal L, \mathcal E, C]$ be general in $\mathcal P^{\perp}_{5,3}$ and $y = r_{5,3}(x)$, then $y = [C, \eta]$ with  $\eta := \mathcal E \otimes \mathcal O_C$. Let  $y  \in \mathcal M$ be general, we prove that a unique $x = [S, \mathcal L, \mathcal E, C]$ exists so that $[C, \mathcal E \otimes \mathcal O_C]= y$.  We already know, for a general $y = [C, \eta] \in \mathcal M$, the surjectivity of the multiplication map   
$$
\mu_C: H^0(\omega_C \otimes \eta) \otimes H^0(\omega_C \otimes \eta^{-1}) \to H^0(\omega_C^{\otimes 2}),
$$
because this condition is open and non empty on $\mathcal M$. Then, applying to $\mu_C$ the same construction applied to $\mu_S$, one obtains  \begin{equation}
C \subseteq \Sigma := \mathbb P^{11} \cdot (\mathbb P^3 \times \mathbb P^3) \subset \mathbb P^{15}.
\end{equation}
Let $V = H^0(\omega_C \otimes \eta)^*$ and $\mathring V = H^0(\omega_C \otimes \eta^{-1})^*$, here $C$ is bicanonically embedded in $\mathbb P^{11} := \mathbb P({\rm Im} \ \mu_C)^*$ and the inclusion is the Segre embedding $\mathbb P(V) \times \mathbb P(\mathring V) \subset \mathbb P(V \otimes \mathring V)$.  Now the properness of $\Sigma$ is an open condition on $\mathcal M$, not empty under our claim. Then $(\Sigma, \mathcal O_{\Sigma}(1))$ is a polarized K3 surface as above. Since $y = r_{5,3}(x)$ for some $x = [S, \mathcal L, \mathcal E, C]$, the commutative diagram (\ref{NEW}) implies that $[\Sigma, \mathcal O_{\Sigma}(1)]= [S, \mathcal L]$. Therefore $\mu_C$ defines a rational map, sending  $y = [C, \eta] \in \mathcal M$ to $x \in \mathcal P^{\perp}_{5,3}$, which is inverse to $r_{5,3}$.    \end{proof}
\begin{proof}[Proof of the claim] Since each component of $\Sigma$ has dimension $\geq 2$, it suffices to construct one $\mathbb D \in \vert \mathcal O_{\mathbb P^3 \times \mathbb P^3}(1,1) \vert$ so that $\mathbb D \cdot \Sigma = \mathbb D \cdot S$. We choose the hyperplane section 
\begin{equation} \mathbb D = (P \times \mathbb P^3) + (\mathbb P^3 \times \ring P), \end{equation}
where $P$ and $\ring P$ are general planes. Then we have $\mathbb D \cdot S = D + \ring D$, where $D \in \vert H \vert$ and $\ring D \in \vert \ring H \vert$ are smooth, non hyperelliptic curves of genus $3$. We show, only for $D$, that 
\begin{equation} D = \mathbb P^{11} \cdot (P \times \mathbb P^3) \ , \ \ring D = \mathbb P^{11} \cdot (\mathbb P^3 \times \ring P). \end{equation}
The map $p: D \to P$ is the canonical map; we fix on $P$ coordinates $(x) = (x_1:x_2:x_3)$. The map
$\ring p: D \to \mathbb P^3$ is defined by $\vert \omega_D(b) \vert$,
where $\deg b = 2$ and $h^0(\mathcal O_D(b)) = 0$. This implies that $\omega_D(b)$ is very ample, we fix coordinates $(y) = (y_1: \dots: y_4)$ on $\mathbb P^3$. The resolution of $\mathcal O_{\ring p(D)}(1) \cong \omega_D(b)$ is definitely well known, \cite{Ho}. We have the exact sequence
\begin{equation}
0 \to \mathcal O_{\mathbb P^3}(-1)^{\oplus 3} \stackrel{A} \to \mathcal O_{\mathbb P^3}^{\oplus 4} \to \omega_D(b) \to 0,
\end{equation}
$A = (a_{ij})$ being a $4 \times 3$ matrix of linear forms in $(y)$. Then $\ring p(D)$ is a determinantal curve defined by the cubic minors of $A$.  In particular $A$ has rank $3$ on $\mathbb P^3 - \ring p(D)$ and, since $\ring p: D \to \ring p(D)$ is biregular and $\ring p(D)$ is smooth, it also follows that $\ring p(D)$ is the set of points $y \in \mathbb P^3$ such that $A$ has exactly rank $2$.  This  implies that the equations  $ a_{i1}x_1 + a_{i2}x_2 + a_{i3}x_3 = 0, \ i = 1 \dots 4,$ 
define a complete intersection $\hat D \subset P \times \mathbb P^3$ such that $\Supp \hat D = D$. Finally one easily computes that $\hat D$ and $D$ have the same degree $10$ with respect to $\mathcal O_{\mathbb P^3 \times \mathbb P^3}(1,1)$. This implies $\hat D = D$ and the claim follows. \end{proof} 
\subsection{\sf {The case of genus $6$}}  
\begin{theorem} \label{announce} The Mukai map $r_{6,3}: \mathcal P^{\perp}_{6,3} \to \mathcal R_{6,3}$ has not maximal rank. \end{theorem}
 In this paper we only sketch the proof of this theorem and its geometric motivation: see section 7 and also  \cite{Ve1}. We postpone some details to further investigation on $\mathcal R_{6,3}$. We conclude that the mentioned analogies are confirmed for $\ell = 3$:  \it the  Mukai maps 
\begin{equation}  m_{_{11 \pm 1}} \ , \ r_{_{7 \pm 1,  2}} \ , \ r_{_{5 \pm 1,  3}} \end{equation} have not maximal rank, while they are birational for $g = 11, \ 7, \ 5$. \rm These maps are not dominant  for $g = 10, \ 6, \ 4$ and they have positive dimensional fibre for $g = 12, \ 8, \ 6$.

 \section{Views on Fano threefolds with sections of level $2$ or $3$}
We close this paper discussing some families of Fano threefolds $\overline X \subset \mathbb P^{g+1}$, whose general hyperplane sections are singular K3 surfaces $\overline S$ of the considered types.
Then  $\overline S$ is endowed with a degree $\ell$ cyclic cover $\pi: \tilde S \to \overline S$ with branch locus $\Sing \overline S$. Moreover its minimal desingularization $\nu: S \to \overline S$ fits in a
standard level $\ell$ K3 surface $(S, \mathcal L, \mathcal E)$, so that $\mathcal L \cong \nu^* \mathcal O_{\overline S}(1)$ and $\mathcal E$ induces $\pi: \tilde S \to \overline S$.
We have $\ell = 2, 3$. \par
For some families a natural cyclic cover $\pi_{\overline X}: \tilde X \to \overline X$ is visible, with branch locus the curve $\Sing \overline X$. However we do not address it here.
 The existence of these families implies that $r_{g, \ell}$ has not maximal rank. They correspond to the peculiar values \begin{equation} (g, \ell) = (6,3), (6,2), (8,2), (4,3). \end{equation}
For $\ell = 2$ these families are known, \cite{FV, KLV1, Ve}. The case $(6,2)$ is revisited here with emphasis on a singular quadratic complex of the Grassmannian $G(2,5)$. This implies that $r_{6, 2}$ is not of 
maximal rank.  For $(6,3)$ we introduce a family of Gushel - Mukai threefolds singular along a rational normal sextic curve. This is responsible for the failure of the maximal rank of $r_{6,3}$.  The case $(8,2)$ is similar and not treated here, \cite{Ve}. Finally we point out the plausible relation of  the case $(4,3)$ to the $G_2$-variety. 
\subsection{ \sf A singular Gushel - Mukai threefold: $\ell = 3$ and $g = 6$} We sketch the geometric construction implying theorem (\ref{announce}). Let $g = 6$ and $\ell = 3$, keeping our notation we consider $p \times \ring p: S \to \mathbb P^4 \times \mathbb P^4$.   Then $p$  is defined by the linear system
\begin{equation}
\vert H \vert = \vert C - \frac 13 \sum_{i = 1 \dots 6} (B_i+ 2B_i') \vert,
\end{equation}
where $B_i+ B_i'$, are the connected components of $B_{\mathsf {red}}$. Let $x_0 := [S, \mathcal L, \mathcal E, C] \in \mathcal P^{\perp}_{6,3}$ be a general point, then a standard analysis shows that $p: S \to p(S)$ is the contraction of $\sum B_i$ to six points and that $p(B_i')$ is a line.  Moreover we have
\begin{equation}
p(S) = F_0 \cap Q,
\end{equation}
where $F_0$ is a cubic and $Q$ a \it smooth \rm  quadric.  Notice that  $p \vert C$ is the embedding defined by $\omega_C \otimes \eta^{-1}$,  since $CB_i = 0$ then $p(C) \cap \Sing p(S) = \emptyset$. Let $C' := p(C)$ and let
 \begin{equation}
0 \to \mathcal I_{p(S)}(3) \to \mathcal I_{p(C)}(3) \to \mathcal I_{C' \vert p(S)}(3) \to 0
\end{equation} 
be the standard exact sequence of ideal sheaves of $Q$,   we notice the isomorphisms $\mathcal I_{p(S)}(3) \cong \mathcal O_Q$ and $p_*: H^0(\mathcal O_S(3H-C)) \to H^0(\mathcal I_{p(C) \vert p(S)}(3))$. This implies that
\begin{equation}
0 \to H^0( \mathcal O_Q) \to H^0(\mathcal I_{C'}(3)) \to H^0(\mathcal O_S(3H - C)) \to 0
\end{equation}
is its associated long exact sequence. It easily follows that $C'$ is projectively normal. 
A second standard step is the remark that $\mathcal O_S(3H-C)$ is a genus $3$ polarization of $S$.  Now let $M \in \vert 3H - C \vert$, then $p_*(C+M) \in \vert \mathcal I_{p(C) \vert p(S)}(3) \vert$ and
it is cut on $p(S)$ by a cubic hypersurface. Therefore we have in $Q$ the complete intersection scheme
\begin{equation}
p_*(C + M) = F_0 \cap F_{\infty} \cap Q,
\end{equation}
where $F_0, F_{\infty}$ are cubics. Let $S'_0 = F_0 \cdot Q$ and $S'_{\infty} = F_{\infty} \cdot Q$. We consider the pencil
\begin{equation}
P_M = \lbrace S'_t, \ t \in \mathbb P^1 \rbrace,
\end{equation}
of cubic sections of $Q$ generated by $S'_0$ and $S'_{\infty}$. We can assume $p(S) = S'_0$,  notice that a general $S'_t$ is a possibly singular $K3$ surface, smooth along $C'$.
Let $\sigma_t: S_t \to S'_t$ be its minimal desingularization and $C_t := \sigma^*_t C'$, then $S_t$ is endowed with the line bundles 
\begin{equation}
\mathcal H_t := \sigma^*_t \mathcal O_Q(1),  \ \mathcal L_t := \mathcal O_{S_t}(C_t), \ \mathcal E_t := \mathcal L_t \otimes \mathcal H^{-1}_t.
\end{equation}  
For $t = 0$ the fourtuple $(S_t, \mathcal L_t, \mathcal E_t, C_t)$ defines the point $x_0 = [S, \mathcal L, \mathcal E, C]$ of $\mathcal P^{\perp}_{6,3}$. For $t \neq 0$ we have
constantly $C_t = C$. Now consider the family of fourtuples
\begin{equation}
\lbrace (S_t, \mathcal L_t, \mathcal E_t, C_t), \ t \in \mathbb P^1 \rbrace,
\end{equation}
then the assignment $t \to [\mathcal S_t, \mathcal L_t] \in \mathcal F_6$ defines a non constant rational map $m: \mathbb P^1 \to \mathcal F_6$.  Assume $(S_t, \mathcal L_t, \mathcal E_t)$ is a K3 surface of level $3$ for a general $t$. Then $m$ lifts to a map $\tilde m: \mathbb P^1 \to \mathcal P^{\perp}_{6,3}$, sending $t$ to $[S_t, \mathcal L_t, \mathcal E_t, C_t]$, and the next statement immediately follows.
\begin{proposition}  If $(S_t, \mathcal L_t, \mathcal E_t)$ is a K3 surface of level $3$ for a general $t$,  the curve $\tilde m(\mathbb P^1)$ is in the fibre at the point $[C, \eta]$ of the Mukai map $r_{6,3}$, which is therefore not of maximal rank.
\end{proposition}
The assumption mentioned in the statement depends on the choice of the element $M$ in $ \vert 3H - C \vert$ and in general it is not satisfied.   However the assumption is satisfied choosing in $\vert M \vert$ the very special element
\begin{equation}
M_0 := 2A + \sum_{i = 1 \dots 6} B_i,
\end{equation}
where $A$ is the \it unique \rm element of $\vert C -  \sum_{i = 1 \dots 6} (B_i+B_i') \vert$.  The curve $A$ is  biregular to $\mathbb P^1$ and $p \vert A$ embeds it as a rational normal quartic curve.
 Let $A' = p(A)$, then the base scheme of $P_{M_0}$ is a non reduced, complete intersection curve and its $1$-cycle is
 \begin{equation}
 p_*(M_0 + C) = 2A' + C'.
 \end{equation}
 In other words  the surfaces $S'_t$ intersect along a contact curve $A'$ of multiplicity two and along $C'$. It turns out that a general  \it $\Sing S'_t$ consists of six nodes \rm moving in $A'$  and each node belongs to a line in $S'_t$. This can be shown using the special property that $\eta \cong \omega_{C'}(-1) \in \Pic C$ is of $3$-torsion. Omitting further details of this construction, let us just say that $M_0$ defines a pencil of level $3$ and genus $6$ K3 surfaces as required. \par
 To close geometrically this sketch let $\mathsf A$ be the non reduced component, supported on $A'$, of the base curve of $P_{M_0}$ and $\mathcal I_{\mathsf A \vert Q}$ its ideal sheaf. Consider the rational map  
 \begin{equation} 
 \phi: Q \to \mathbb P^7
 \end{equation}
 defined by the  linear system $\vert \mathcal I_{\mathsf A \vert Q}(3) \vert$. Let us notice the following property.
 \begin{proposition}  The map   $\phi$ is birational onto its image $W$, which is a singular Gushel - Mukai threefold whose general hyperplane sections are singular $K3$ surfaces $\overline S$ as above. \end{proposition}
Therefore $W$ is a complete intersection of type $(1,1,2)$ in the Grassmannian $G(2,5)$. We notice that $\Sing W$ is a rational normal sextic curve. This completes our sketch.

 \subsection{{ \sf The tangential quadratic complex of $\mathbb P^4$: $\ell = 2$ and $g = 6$}}
Let $\mathbb G_n$ be the Pl\"ucker embedding of the Grassmannian of lines of $\mathbb P^n$, a quadratic complex is just a quadratic section of $\mathbb G_n$. Let $Q \subset \mathbb P^n$ be
a quadric, then the family $\mathbb T$ of tangent lines to $Q$ is a quadratic complex, named sometimes the \it tangential quadratic complex. \rm We assume $Q$ is smooth, then $\mathbb T$ is a Fano variety.    Notice that $\Sing \mathbb T$ is the Hilbert scheme of lines of $Q$, of codimension and multiplicity $2$ in $\mathbb T$. \par
Now we assume $n$ is even. Then $\mathbb T$ has a unique  nontrivial   quasi \'etale 2:1 cover
\begin{equation}
\pi: \tilde {\mathbb T} \to \mathbb T,
\end{equation}
whose branch locus is $\Sing \mathbb T$. Let us describe the known map $\pi$ in the case $n = 4$, since it is linked to the Mukai map $r_{6,2}: \mathcal P^{\perp}_{6,2} \to \mathcal R_6$ and its behavior. This is treated in \cite{FV}. 
For $n = 4$ the Hilbert scheme of lines of $Q$ is the $2$-Veronese embedding of $\mathbb P^3$, say
\begin{equation}  V \subset \mathbb G_4 \subset \mathbb P^9. \end{equation}
Let $t \in \mathbb T$, consider the pencil $\lbrace H_p, p \in t \rbrace$, where $H_p$ is the  polar hyperplane to $Q$ at $p$. Its base locus is a plane $P_t$ and $Q_t := P_t \cdot Q$ is a conic. Since $t$ is tangent to $Q$, a standard exercise shows that $\Sing Q_t = t \cap Q$. This defines a smooth, integral correspondence
\begin{equation} \tilde {\mathbb T} := \lbrace (t,r) \in \mathbb T \times V \ \vert \ r \subset Q_t \rbrace. \end{equation}
Notice that its projection onto $\mathbb T$ is a quasi \'etale $2:1$ cover branched on $V$, say  
\begin{equation}
\pi: \tilde {\mathbb T} \to \mathbb T.
\end{equation}
Indeed the fibre $\zeta_t := \pi^*(t)$ is the Hilbert scheme of lines of $Q_t$ and is finite of length $2$.  Then  $\zeta_t$ is smooth iff $\rank Q_t = 2$ iff $t \notin V$ and $\zeta_t$ has multiplicity $2$ iff $\rank Q_t = 1$ iff
$t \in V$. \par Now it is well known that a general $2$-dimensional linear section $ \overline S = \mathbb T \cap \mathbb P^6$ is the model defined by $\vert \mathcal L \vert$ of $S$,  where $[S, \mathcal L, \mathcal E] \in \mathcal F^{\perp}_6$ is general. In particular $\Sing \overline S = V \cap \mathbb P^6 $ is an even set of $8$ nodes, defining $\pi \vert \tilde S$ with $\tilde S = \pi^{-1}(\overline S)$, cfr. \cite{FV, KLV1, KLV2}. For
$\ell = 2$ and $[S, \mathcal L, \mathcal E] \in \mathcal F^{\perp}_g$, the surface $S$, or its model $\overline S$, is known as a standard Nikulin surface of genus $g$. Therefore we can say that a general $3$-dimensional linear section of $\mathbb T$ is \it a Fano threefold whose hyperplane sections are standard Nikulin surfaces of genus $6$. \rm Let us denote such a section  by 
\begin{equation}
X = \mathbb T \cap \mathbb P^7,
\end{equation}
notice that $\Sing X$ is a curvilinear section of $V$, hence an elliptic curve of degree $8$. \par
Finally let $\mathcal C$ and $\overline {\mathcal S}$ respectively be the family of general curvilinear sections $C$ and that of general $2$-dimensional linear sections $\overline S$ of $\mathbb T$.
Consider the family of pairs
\begin{equation}
\mathcal P := \lbrace(C, \overline S) \in \mathcal C \times \overline {\mathcal S} \ \vert \ C \subset \overline S \rbrace.
 \end{equation}
Let $(C, \overline S) \in \mathcal P$ then $C$ is a canonical curve and $ C \in \vert \mathcal O_{\overline S}(1) \vert$. Let $\nu: S \to \overline S$ be the desingularization then $\nu^*C \in \vert \mathcal L \vert $ and $\eta := \mathcal E \otimes \mathcal O_{\nu^*C}$ defines $\pi \vert \tilde C$, where $\tilde C = \pi^{-1}(C)$. Then the assignment of $(C, \overline S)$ to $[S, \mathcal L, \mathcal E, \nu^*C]$ defines a \it dominant \rm  rational map
 $$ m: \mathcal P  \to \mathcal P^{\perp}.$$
We already know  that the Mukai map $r_{6,2}$ fails to be of maximal rank. However we can now see this fact from a geometric
perspective: the existence of the Fano variety $\mathbb T$ and its quasi finite $2:1$ cover $\pi$.  Indeed this implies that $C \in \mathcal C$ is contained in a higher dimensional family of sections
$\overline S$  of $\mathbb T$, so that $C$ cannot have general moduli. \par More precisely  the parameter space $\mathcal C$ is open in the Grassmannian $G(5, 9)$, hence $\dim \mathcal C = 24$.  Moreover $\Aut Q \subset \Aut \mathbb P^4$ has dimension $10$ and acts faithfully on $\mathcal C$. Then we have $\dim \mathcal C \dslash \Aut Q = 14 < \dim \mathcal R_6 = 15$. Hence $r_{6,2}$ cannot be dominant.  \begin{remark} \rm Let $C \in \mathcal  C$ then $\tilde C = \pi^{-1}(C)$ is a smooth, integral curve of genus $11$. We have $\tilde C \subset \tilde S \subset \tilde X \subset \mathbb P^{12}$, where $\tilde X = \pi^{-1}(X)$ is a non prime Fano threefold of genus $11$. We just mention that $\tilde C$ is the base locus of a pencil of hyperplane sections of $\tilde X$ and that the birational Mukai map $m_{11}: \mathcal P_{11} \to \mathcal M_{11}$ is not invertible at $[\tilde C]$.
\end{remark}
  \subsection{ { \sf The $\mathsf G_2$-variety: $\ell = 3$ and $g = 4$} }
  A \it geometric interpretation \rm seems plausible and it is possibly postponed to future work. It relates to the failure of the Mukai map in genus $10$.
As in (\ref{diagram}) let $\pi: \tilde S \to \overline S$ be the cover induced by $\mathcal E$ and $\nu: S \to \overline S$ the desingularization map.
For a general $C$ the map $\nu: C \to \overline S \setminus \Sing \overline S$ is an embedding, then we set $C := \nu(C)$. Let $\tilde C := \pi^{-1}(C)$ then $(\tilde S, \mathcal O_{\tilde S}(\tilde C ))$ is a K3 surface of genus $10$. This suggests that $\tilde S$ embeds in the $G_2$-variety $W \subset \mathbb P^{13}$ as a linear section, \cite{Mu}. Now a general curvilinear section of $W$ is not general as a genus $10$ curve. In the same way, if it is a triple cover of a genus $4$ curve, it seems not a general genus $4$ triple cover.  
{\tiny
     }

 \end{document}